\documentclass[11pt]{article}

\pdfoutput=1

\usepackage{graphics,epsfig,graphicx,color}
\graphicspath{{/EPSF/}{../Figures/}{Figures/}}
\usepackage{subfig}

\usepackage{amssymb,latexsym}

\usepackage{setspace}

\usepackage{amsmath}

\usepackage{mathrsfs,amsfonts}

\usepackage{amsthm,amsxtra}

\usepackage[final]{showkeys}

\usepackage{stmaryrd}

\usepackage{algorithm}
\usepackage{algorithmicx}
\usepackage{algpseudocode}

\newif\ifPDF
\ifx\pdfoutput\undefined
	\PDFfalse
\else
	\ifnum\pdfoutput > 0
		\PDFtrue
	\else
		\PDFfalse
	\fi
\fi

\ifPDF
	\usepackage[debug,pdftex,colorlinks=true, 
	linkcolor=blue, bookmarksopen=false,
	plainpages=false,pdfpagelabels]{hyperref}
\else
	\usepackage[dvips]{hyperref}
\fi

\usepackage{fancyhdr}

\newtheorem{theorem}{Theorem}[section]
\newtheorem{lemma}[theorem]{Lemma}

\newtheorem{remark}[theorem]{Remark} 
\newtheorem{corollary}[theorem]{Corollary}

\newcommand{\diag}{{\rm diag}}

\newcommand{\eps}{\varepsilon}

\newcommand{\sfc}{\mathsf c}
\newcommand{\sfC}{\mathsf C}

\newcommand{\bbR}{\mathbb R} \newcommand{\bbS}{\mathbb S}

\newcommand{\bzero}{{\mathbf 0}}

\newcommand{\bnu}{{\boldsymbol \nu}}

\newcommand{\bxi}{\boldsymbol \xi}

\newcommand{\bk}{\mathbf k} 
 \newcommand{\bn}{\mathbf n}
 \newcommand{\bp}{\mathbf p}
\newcommand{\bq}{\mathbf q} 
\newcommand{\bs}{\mathbf s} 
\newcommand{\bu}{\mathbf u} \newcommand{\bv}{\mathbf v} 
 \newcommand{\bx}{\mathbf x} 
 \newcommand{\bz}{\mathbf z}

 \newcommand{\bP}{\mathbf P} 
\newcommand{\bQ}{\mathbf Q}

\newcommand{\cA}{\mathcal A} 
  
 \newcommand{\cF}{\mathcal F}
 \newcommand{\cH}{\mathcal H}
 
 \newcommand{\cL}{\mathcal L}
 
\newcommand{\cO}{\mathcal O}

 \newcommand{\cV}{\mathcal V}
 \newcommand{\cX}{\mathcal X} 
\newcommand{\cY}{\mathcal Y} 

\newcommand{\hN}{\frac{(N+1)}{2}}

\DeclareMathOperator*{\argmin}{arg\,min}

\newenvironment{keywords}
{\noindent{\bf Key words.}\small}{\par\vspace{1ex}}

\newcommand{\note}[1]{#1}

\title{Quantitative PAT with simplified $P_N$ approximation}

\author{
    Hongkai Zhao\thanks{
            Department of Mathematics, Duke University, NC 27705, zhao@math.duke.edu .
    }
    \and
    Yimin Zhong\thanks{
    		Department of Mathematics, Duke University, NC 27705, yimin.zhong@duke.edu .
	}
}
    
\date{}

\begin{document}

\maketitle



\begin{abstract}
    The photoacoustic tomography (PAT) is a hybrid modality that combines the optics and acoustics to obtain high resolution and high contrast imaging of heterogeneous media. In this work, our objective is to study the inverse problem in the quantitative step of PAT which aims to reconstruct the optical coefficients of the governing radiative transport equation from the ultrasound measurements. In our analysis, we take the simplified $P_N$ approximation of the radiative transport equation as the physical model and then show the uniqueness and stability for this modified inverse problem. Numerical simulations based on synthetic data are presented to validate our analysis.
\end{abstract}






\section{Introduction}
\label{SEC:Intro} 
The photoacoustic tomography (PAT)~\cite{bal2013hybrid,beard2011biomedical,cox2009challenges,kuchment2007mathematics,wang2017photoacoustic,scherzer2010handbook, ren2015inverse, ren2013quantitative} is an emerging hybrid imaging modality that reconstructs high resolution images of optical properties of heterogeneous media. The PAT experiment uses a pulse of near-infra-red (NIR) laser into the medium of interest (e.g. fat, bone, tumor tissues). These photons propagate inside the medium  by following the radiative transport process.  During the propagation, a portion of the photons is absorbed by the medium and then converted into heat which causes a local thermoelastic expansion. Such expansion induces a transient pressure change and leads to the propagation of ultrasound. The ultrasound signals are measured around the boundary of the medium and we need to infer the optical properties from the acoustic measurements.

The photon transport process is usually modeled by radiative transport equation. Let $X =\Omega\times\bbS^2$, where $\Omega$ is the physical domain and $\bbS^2$ denotes the unit sphere in 3D, the photon density function $u(\bx ,\bv)$ satisfies the following
\begin{equation}\label{EQ:RTE}
\begin{aligned}
\bv \cdot \nabla u(\bx, \bv) + \sigma_t(\bx) u(\bx, \bv) &=\sigma_s(\bx) \int_{\bbS^2} p(\bv\cdot\bv') u(\bx, \bv') d\bv'\quad&\text{ in }&X\,, \\ 
u(\bx, \bv) &= f(\bx, \bv)\quad &\text{ on }&\Gamma_{-}\,,
\end{aligned}
\end{equation}
where $\Gamma_{-} = \{ (\bx, \bv)\in \partial\Omega\times\bbS^{d-1}\mid -\bnu(\bx)\cdot \bv > 0\}$ is the incoming boundary set.  $\sigma_s, \sigma_t$ are the scattering and total absorption coefficients respectively, $\sigma_a:= \sigma_t - \sigma_s$ is the intrinsic absorption coefficient. $f(\bx, \bv)$ is the external illumination source. The scattering phase function $p(\bv\cdot\bv')$ is usually chosen as the Henyey-Greenstein function
\begin{equation}
p(\cos\theta) = \frac{1}{4\pi} \frac{1-g^2}{(1 + g^2 - 2g\cos\theta)^{3/2}}\,,
\end{equation}
where $g\in (-1,1)$ is the anisotropy parameter. 

The energy absorbed by the medium is $\sigma_a \int_{\bbS^2} u(\bx, \bv) d\bv$, then the initial pressure field generated by the \emph{photoacoustic effect} is:
\begin{equation}
    H(\bx) := \Upsilon(\bx) \sigma_a(\bx) \int_{\bbS^2} u(\bx, \bv) d\bv,
\end{equation}
where $\Upsilon(\bx)$ is the dimensionless Gr\"uneisen coefficient which measures the efficiency of the photoacoustic effect.

Then the initial pressure field $H(\bx)$ propagates the ultrasound wave, which satisfies the following equation~\cite{stefanov2009thermoacoustic}:
\begin{equation}
    \begin{aligned}
        \frac{1}{c^2(\bx)} \frac{\partial^2 p(\bx, t)}{\partial t^2} - \Delta_{\bx} p(\bx, t) &= 0,\quad &\text{ in }& \bbR^3\times [0, \infty),\\
        p(\bx, 0) = H(\bx),\quad \frac{\partial p}{\partial t}(\bx, 0) &= 0, \quad &\text{ in }&\bbR^3.
    \end{aligned}
\end{equation}
Here $c(\bx)$ is the wave speed of the underlying medium. The measured acoustic signals are $p(\bx ,t)$ on $\partial\Omega\times [0, T]$ for sufficient large observation time $T$. 

The usual reconstruction of PAT is a two-step process. The first step is to reconstruct the initial pressure field $H(\bx)$ from the ultrasound measurements. This problem has been studied extensively by~\cite{agranovsky2007reconstruction,ammari2012photoacoustic,stefanov2009thermoacoustic,haltmeier2005filtered,hristova2009time} and the references therein. Here we assume this step has been finished and recovered the initial pressure field $H(\bx)$ and we focus on the second step to reconstruct the optical properties $(\sigma_a, \sigma_s, \Upsilon)$ from the quantity $H(\bx)$. Under the diffusion approximation, this quantitative PAT (qPAT) problem has been well studied~\cite{ding2015one,ren2013hybrid,bal2010inverse,bal2011multi}. However, with the radiative transport equation~\eqref{EQ:RTE}, the multi-source inverse problem theory has not been well established except for \emph{albedo} type data~\cite{mamonov2012quantitative,bal2010inverse}, which requires infinitely many angularly resolved illumination sources $f(\bx, \bv)$. The reconstruction of only absorption coefficient $\sigma_a$ has been recently considered in~\cite{ren2020unique} for nonlinear setting. It is still unclear about the uniqueness and stability of the reconstructions of $(\sigma_a, \sigma_s, \Upsilon)$ with finite many source functions or angularly independent sources.

In this paper, we aim to study the qPAT problem with the simplified $P_N$ (\note{$N$ being an odd integer}) approximation to the equation~\eqref{EQ:RTE} with angularly independent source functions, that is, $f(\bx, \bv) = f(\bx)$. The simplified $P_N$ approximation is also referred as $SP_N$ method, which is utilized to solve the radiative transport equation by forming a system of elliptic equations~\cite{mcclarren2010theoretical,klose2006simplified}. The $SP_N$ approximation with relatively small $N \le 7$ has been applied to many optical imaging methods~\cite{wright2006reconstruction, chu2009light,klose2006light} and outperforms the traditional diffusion approximation ($P_1$ method). Theoretically, the simplified $P_N$ approximation is derived from the $P_N$ formulation~\cite{gelbard1960application} and the $P_N$ approximation converges to the exact solution of RTE as $N\to\infty$. Under appropriate conditions $SP_N$ and $P_N$ are equivalent, see~\cite{mcclarren2010theoretical}, however in general, they are different and not necessarily converging to the same limit. For the qPAT problem, the case with $N=3$ has been considered in~\cite{frederick2018image}, in our work, we extend the theory to arbitrary order $N$ under a unified framework.  

 Under the $SP_N$ approximation, the RTE's solution is expanded with Legendre polynomials, which derives a weakly coupled diffusion equation system~\cite{klose2006simplified}. Formally, the 3D $SP_N$ approximation takes the 1D $P_N$ equations and replace the diffusion operators with 3D's counterpart, which is
\begin{equation}
\begin{aligned}
&-\left(\frac{n+1}{2n+1}\right)\nabla \frac{1}{\sigma_{n+1}} \nabla \left[\left(\frac{n+2}{2n+3}\right)\phi_{n+2} +\left( \frac{n+1}{2n+3}\right)\phi_n\right] \\&\quad - \left(\frac{n}{2n+1}\right)\nabla \frac{1}{\sigma_{n-1}} \nabla \left[\left(\frac{n}{2n-1}\right)\phi_{n} +\left( \frac{n-1}{2n-1}\right)\phi_{n-2}\right] + \sigma_n \phi_n = 0
\end{aligned}
\end{equation}
for $n=\note{0}, 2,4,\dots,N-1$, where coefficients $\sigma_{n} = \sigma_a + \sigma_s(1 - g^n)$, $\sigma_0 = \sigma_a$. \note{The system is closed by seting $\phi_n$ as zeros for $n\ge N+1$ and $n \le-2 $ and only consists of even-indexed $\phi_n$.} 
Physically speaking, $\phi_n$ represents the $n$-th Legendre moments of the solution $u(\bx, \bv)$ and $\phi_0$ will be the angular average of the solution.
The corresponding mixed boundary conditions are derived from the 1D $P_N$ equation's boundary conditions by replacing $\frac{\rm{d}}{\rm{d}x}$ with $\bn\cdot \nabla$ on the boundary~\cite{klose2006simplified}.

In the following context, we let 
\begin{equation}
\begin{aligned}
\varphi_{n} = (2n-1)\phi_{2n-2} + (2n)\phi_{2n},\quad n = 1,2,\dots, (N+1)/2\,,
\end{aligned}
\end{equation}
and the column vector $\Phi = [\varphi_1,\varphi_2,\dots, \varphi_{(N+1)/2}]^{\text{T}}$, which satisfies
\begin{equation}\label{EQ:T MAT}
\Phi = \begin{pmatrix}
\varphi_1 \\ \varphi_2 \\\varphi_3 \\ \vdots\\ \varphi_{(N+1)/2} 
\end{pmatrix} = M \begin{pmatrix}
\phi_0 \\ \phi_2 \\\phi_4\\ \vdots\\ \phi_{N-1} 
\end{pmatrix}\text{ where } M :=\begin{pmatrix}
1 & 2 & 0 & \dots &0 \\
0 & 3 & 4 & \dots &0 \\
0 & 0 & 5 & \dots &0 \\
\vdots & \vdots & \vdots &  \ddots & \vdots\\
0 & 0 & 0 & \dots & N 
\end{pmatrix}\,,
\end{equation}
since the matrix $M$ is upper triangular, its inverse is also upper triangular, let the row vector $\bs_k = [s_{k,1},\dots, s_{k, (N+1)/2}]$ represent the $k$-th row of the inverse matrix $M^{-1}$, according to Lemma~\ref{LEM: M INV}, its entries are
\begin{equation}
s_{k, l} = 
\begin{cases}
\frac{1}{2k-1}(-1)^{l-k}\frac{((2l-2))!!}{(2l-1)!!}\frac{(2k-1)!!}{(2k-2)!!},\quad &l\ge k\,,\\
0, &\text{otherwise}\,.
\end{cases} 
\end{equation}
Then we derive the $SP_N$ diffusion system
\begin{equation}\label{EQ:SPN}
-\nabla \cdot \frac{A_n}{(4n-1)\sigma_{2n-1}} \nabla \varphi_n - \nabla \cdot \frac{ B_n}{(4n-5)\sigma_{2n-3}}\nabla \varphi_{n-1} + \sigma_{2n-2} \bs_n\cdot \Phi = 0
\end{equation}
for $n = 1,2,\dots, (N+1)/2$, where the constants $A_n, B_n$ are defined by
\begin{equation}\nonumber
\begin{aligned}
A_n &= \frac{2n-1}{(4n-3)},\quad 
B_n = \frac{2n-2}{(4n-3)}.
\end{aligned}
\end{equation}
We remark that the following matrix differs from $M^{\text{T}}$ by only a factor of diagonal matrix,
\begin{equation}
\begin{pmatrix}
A_1 & 0 & 0 & \dots & 0  \\
B_2 & A_2 & 0 &\dots & 0 \\
0 & B_3 & A_3 & \dots & 0 \\
\vdots & \vdots & \ddots & \ddots & \vdots  \\
0 & 0 & 0 & B_{(N+1)/2} & A_{(N+1)/2} 
\end{pmatrix} =L M^{\text{T}}\;\text{ where }\; L:=\begin{pmatrix}
1 & 0 & 0 & \dots & 0  \\
0 & \frac{1}{5} & 0 &\dots & 0 \\
0 &0 & \frac{1}{9} & \dots & 0 \\
\vdots & \vdots & \ddots & \ddots & \vdots  \\
0 & 0 & 0 & 0 & \frac{1}{2N-1}
\end{pmatrix} \,.
\end{equation}
The corresponding mixed boundary conditions are 
\begin{equation}\label{EQ:BD}
\sum_{n=1}^{(N+1)/2} \mu_{2m-1, 2n-2} \phi_{2n-2} + \frac{1}{(4m-1)\sigma_{2m-1}} \frac{\partial \varphi_m}{\partial \bn} = k_{2m-1} f,\quad m=1,2,\dots,(N+1)/2\,.
\end{equation}
The constants $\mu_{2m-1,2n-2}$ and $k_{2m-1}$ are
\begin{equation}\nonumber
\begin{aligned}
\mu_{2m-1, 2n-2} &= (4n-3)\int_0^1 P_{2m-1}(x)P_{2n-2}(x) d x \\&=(-1)^{m+n-1} \frac{\Gamma(m+\frac{1}{2})\Gamma(n-\frac{1}{2})}{\pi \Gamma(m)\Gamma(n) (m+n-1)} \frac{(4n-3)}{2n-2m -1 }\,,\\
k_{2m-1} &= \note{\int_0^1 P_{2m-1}(x) dx =} (-1)^{m-1}\frac{(2m-1)!!}{(2m-1)(2m)(2m-2)!!}\,,
\end{aligned}
\end{equation}
\note{where $P_l$ is the degree $l$ Legendre polynomial with normalization condition $P_l(1) = 1$ and the symbol $!!$ dnotes the double factorial. Since we have assumed $f(\bx, \bv) = f(\bx)$, these boundary condition coefficients are simply obtained through integration with Legendre polynomials on the half sphere (incoming directions) and independent of $f$. } 
For convenience, we also define the following matrices  $\Sigma_e$ and $R$  for later uses, 
\begin{equation}\nonumber
\begin{aligned}
\Sigma_e &= \begin{pmatrix}
\sigma_0 & 0 & 0 & \dots & 0  \\
0 & \sigma_2 & 0 &\dots & 0 \\
0 & 0 & \sigma_4 & \dots & 0 \\
\vdots & \vdots & \vdots & \ddots & \vdots  \\
0 & 0 & 0 & \dots & \sigma_{N-1} 
\end{pmatrix}
\end{aligned}
\end{equation}
and
\begin{equation}\nonumber
R = (R_{ij})_{i,j=1,2,\dots,(N+1)/2}, \quad \text{with } R_{ij} = \mu_{2i-1, 2j-2}.
\end{equation}
In the quantitative photoacoustic tomography, we suppose the datum $H(\bx) = \Upsilon(\bx)\sigma_a(\bx) \phi_0(\bx)$ has been reconstructed from the measured acoustic signals. In the following sections, we will analyze the uniqueness and stability of reconstruction of the coefficients $\Upsilon, \sigma_a, \sigma_s$ from the internal data $H$. We also make the following general assumptions for the rest of paper.

\begin{enumerate} 
    \item [$\cA$-i] The physical domain $\Omega$ is simply connected with $C^{2,1}$ boundary.
    \item [$\cA$-ii] The coefficients $(\sigma_a, \sigma_s, \Upsilon)$ are non-negative and bounded. There exists constants $\underline{c}$ and $\overline{c}$ that
    \begin{equation}
    0\le \underline{c}\le \sigma_a, \sigma_s, \Upsilon \le \overline{c} <\infty\,.
    \end{equation}
    \item [$\cA$-iii] The coefficient $\sigma_a, \sigma_s\in C^{1,1}(\overline{\Omega})$. There exists a constant $M_0$ that $$\|\sigma_a\|_{C^{1,1}(\Omega)}, \|\sigma_s\|_{C^{1,1}(\Omega)}\le M_0 < \infty\,.$$ Moreover,  $\sigma_a, \sigma_s$ are both known on $\partial\Omega$.
    \item [$\cA$-iv] The boundary source function $f\in H^{5/2}(\partial\Omega).$
   \end{enumerate} 
   The rest of the paper is organized as follows. We first present in Section~\ref{SEC:Prop} some general
   properties of the forward problem with $SP_N$ approximation. Then in Section~\ref{SEC:Recon} we consider the reconstruction of a single
   coefficient from a single data set $H(\bx)$ and the reconstruction of two coefficients simultaneously
    with multiple data in linearized settings. We then demonstrate some numerical simulations
   based on synthetic data in Section~\ref{SEC:Num} to validate some of our theoretical results. Conclusions are found in Section~\ref{SEC:Concl}.
\section{General properties}
\label{SEC:Prop}
For the forward problem, we establish the wellposedness for the $SP_N$ approximation. In order to show there exists a unique weak solution $\Phi \in [H^1(\Omega)]^{(N+1)/2}$ for~\eqref{EQ:SPN} and~\eqref{EQ:BD}, we only have to consider the corresponding variational form for the diffusion system. \note{By rewriting~\eqref{EQ:SPN} and~\eqref{EQ:BD} into the matrix form, 
\begin{equation}\label{EQ: WEAK}
    \begin{aligned}
        -\nabla \cdot (LM^{\text{T}} D \nabla \Phi) + \Sigma_e M^{-1}\Phi &= \bzero,\quad &\text{ in }& \Omega, \\
        R M^{-1}\Phi + D \frac{\partial \Phi}{\partial\bn} &= \bk f,\quad &\text{ on }&\partial\Omega, 
    \end{aligned}
\end{equation}
where $D$ is a diagonal matrix with elements $D_{nn} = \frac{1}{(4n-1)\sigma_{2n-1}}$, $n=1,2,\dots, (N+1)/2$ and $\bk f$ is a vector with $n$-th element as $k_{2n-1} f$.}
Let $\Psi = [\psi_1, \psi_2,\dots, \psi_{(N+1)/2}]^T\in [H^1(\Omega)]^{(N+1)/2}$ be a test function vector. \note{ Multiply the matrix form~\eqref{EQ: WEAK} with the vector $L^{-1}M^{-1}\Psi$ and integrate over $\Omega$}, then the weak form of the $SP_N$ system is
\begin{equation}\label{EQ:BILINEAR}
\begin{aligned}
B(\Phi, \Psi) &:= \sum_{n=1}^{(N+1)/2} \int_{\Omega} \frac{1}{(4n-1)\sigma_{2n-1}} \nabla \varphi_n\cdot \nabla \psi_n d\bx  \\ &\quad + \int_{\Omega} \langle  M^{-\text{T}}L^{-1}\Sigma_e M^{-1}\Phi, \Psi \rangle d\bx
+ \int_{\partial\Omega} \langle R M^{-1} \Phi, \Psi \rangle d\bs \\
&= \cL(f, \Psi)\,,
\end{aligned}
\end{equation} 
where $B(\cdot, \cdot)$ is a bilinear form, $\cL$ is a linear functional only involving boundary integrals that
\begin{equation}\label{EQ: LOAD}
\cL(f, \Psi) := \sum_{n=1}^{(N+1)/2}k_{2n-1} \int_{\partial\Omega}  f \psi_n d\bs\,.
\end{equation}
We prove the following property of the bilinear form $B(\cdot,\cdot)$.
\begin{theorem}
	The bilinear form~\eqref{EQ:BILINEAR} is bounded and strictly coercive for any $SP_N$ approximation.
\end{theorem}
\begin{proof}
	The boundedness is obvious since $L$, $M$ are both invertible matrices. We only need to prove the coerciveness. In the following, we will show that the matrices $M\Sigma_e^{-1} LM^{\text{T}}$ and $M^T R$ are positive definite. For $M\Sigma_e^{-1} LM^{\text{T}}$, it is obvious since the diagonal matrix $\Sigma_e^{-1} L$ has all positive entries. For the matrix $M^{\text{T}} R$, we compute its $(i,k)$-th entry by
	\begin{equation}
	\begin{aligned}
	&(2i-1)R_{i,k} + (2i-2)R_{i-1,k} \\
	= \;&(2i-1)(4k-3)\int_{0}^1 P_{2i-1}(x)P_{2k-2}(x) dx + (2i-2)(4k-3)\int_0^1 P_{2i-3}(x)P_{2k-2} (x) dx \\
	=\;& (4i-3) (4k-3) \int_{0}^1 x P_{2i-2}(x) P_{2k-2}(x) dx \\
	=\;&\int_{0}^1 q_{i}(x) q_k(x) dx\quad \text{with } q_i(x)=(4i-3) \sqrt{x} P_{2i-2}(x)\,,
	\end{aligned}
	\end{equation}
	where $P_k$ is the $k$-th Legendre polynomial and we have used the recurrence relation
	\begin{equation}
	(2i-1) P_{2i-1}(x) + (2i-2) P_{2i-3}(x) = (4i-3) x P_{2i-2}(x)\,.
	\end{equation}
	Hence $M^{\text{T}}R$ is semi-positive definite. On the other hand, if there is a vector $\bz = [z_1,\dots, z_{(N+1)/2}]^{\text{T}}\in\bbR^{(N+1)/2}$ that
	\begin{equation}
	0 = \bz^{\text{T}} (M^{\text{T}} R) \bz = \int_{0}^1 x \left(\sum_{k=1}^{(N+1)/2} z_k (4k-3)P_{2k-2}(x)\right)^2 dx\,,
	\end{equation}
	then for any $x \in [0, 1]$, the following polynomial must vanish,
	\begin{equation}
	\sum_{k=1}^{(N+1)/2} z_k (4k-3)P_{2k-2}(x) = 0\,.
	\end{equation}
	Hence the polynomial equals zero for any $x\in \bbR$ and use the fact $\{ P_k(x)\}_{k=1}^{(N+1)/2}$ forms an orthogonal basis on $[-1,1]$, then $\forall k, z_k = 0$. Therefore $M^{\text{T}} R$ is strictly positive definite, so is $R M^{-1} = M^{-\text{T}}(M^{\text{T}} R) M ^{-1}$.
\end{proof}
The wellposedness immediately derives from the Lax-Milgram theorem, there exists a unique weak solution $\Phi\in [H^1(\Omega)]^{(N+1)/2}$ for arbitrary odd integer $N$.  In fact, using the assumptions $\cA$-i to $\cA$-iv, the regularity theorem of elliptic systems~\cite{mclean2000strongly} implies that the unique solution $\Phi\in [H^3(\Omega)]^{(N+1)/2}$, by the Sobolev embedding, the solution $\Phi \in [C^{1,1/2}(\Omega)]^{(N+1)/2}$.
\section{Reconstruction under $SP_N$ approximation}
\label{SEC:Recon}
Generally speaking, if $\sigma_a$ is not negligible, the inverse problem is highly nonlinear and very challenging. Therefore in the following, we only consider the practical scenario that
$\sigma_a \ll \sigma_s (1 - g)$, which means we can simplify the coefficients $\sigma_n = \sigma_a + (1 - g^n)\sigma_s \simeq (1-g^n)\sigma_s$ for $n\ge 1$ and $\sigma_0 = \sigma_a$. This simplification decouples the coefficients $\sigma_a$ and $\sigma_s$. In particular, if $g = 0$, there is no need to perform such simplification.

\vspace{0.5cm}
\noindent {\bf{Reconstruction of $\sigma_a$ only.}} 
Suppose the coefficients $\Upsilon, \sigma_s$ are known on $\Omega$, we consider the reconstruction of $\sigma_a$ from a single measurement datum $H$. Using the assumption that $\sigma_n \simeq (1-g^n)\sigma_s$ for $n\ge 1$, the coefficients $\sigma_n$ are all known for $n\ge 1$. Since $\sigma_a = \sigma_0$, then using $H(\bx) = \Upsilon(\bx) \sigma_a(\bx) \phi_0(\bx)$ and $\phi_0 = \bs_1\cdot \Phi$, we derive that
\begin{equation}\label{EQ:RELATION}
\sigma_0(\bx) = \frac{H(\bx)}{\Upsilon(\bx) \bs_1\cdot \Phi} \,.
\end{equation}
By isolating the term  relevant to $\sigma_0$ ($n=0$) in the bilinear form~\eqref{EQ:BILINEAR}, we can reformulated it as
\begin{equation}\label{EQ:BILINEAR2}
\begin{aligned}
B(\Phi, \Psi) &= \sum_{n=1}^{(N+1)/2} \int_{\Omega} \frac{1}{(4n-1)\sigma_{2n-1}} \nabla \varphi_n\cdot \nabla \psi_n d\bx  + \int_{\Omega} \frac{H(\bx)}{\Upsilon(\bx)\bs_1 \cdot \Phi} (\bs_1\cdot \Phi)(\bs_1 \cdot \Psi) d\bx   \\ &\quad+ \sum_{n=2}^{(N+1)/2} \int_{\Omega}   (4n-3)\sigma_{2n-2} (\bs_{n}\cdot \Phi) (\bs_{n}\cdot \Psi) d\bx
\\&\quad + \int_{\partial\Omega} \langle R M^{-1} \Phi, \Psi \rangle d\bs\,.
\end{aligned}
\end{equation} 
where the row vector $\bs_k$ denotes the $k$-th row of $M^{-1}$. We can establish the following uniqueness and stability result.
\begin{theorem}\label{THM:1}
	Given any $SP_N$ approximation, under the assumptions $\cA$-i to $\cA$-iv and suppose $(\Upsilon, \sigma_s)$ are known, $\sigma_{a,1}$ and $\sigma_{a,2}$ are two admissible absorption coefficients, $H_1, H_2$ are the corresponding internal data, respectively. Then $H_1 = H_2$ implies $\sigma_{a,1} = \sigma_{a,2}$ and the following stability estimate holds 
	\begin{equation}
	\begin{aligned}
	\|(\sigma_{a,1} -\sigma_{a,2})\frac{H_1}{\sigma_{a,1}\Upsilon}\|_{L^2(\Omega)} \le C  \left\|(H_1 - H_2)/\Upsilon\right\|_{L^2(\Omega)}\,,
	\end{aligned}
	\end{equation}
	where $C = C(N, \Omega)$ is a positive constant depending on $N$ and $\Omega$ only.
\end{theorem}
\begin{proof}
	Let $\Phi$ and $\widetilde{\Phi}$ be the weak solutions to the $SP_N$ system for the absorption coefficients $\sigma_{a,1}$ and $\sigma_{a,2}$, respectively. Let $\delta \Phi := \Phi - \widetilde{\Phi} = [\delta\varphi_1,\dots,\delta \varphi_{(N+1)/2}]^{\text{T}}$, then from the bilinear form~\eqref{EQ:BILINEAR2}, we obtain the equation
	\begin{equation}
	\widetilde{B}(\delta \Phi, \Psi) = -\int_{\Omega} \frac{H_1 - H_2}{\Upsilon}(\bs_1\cdot \Psi) d\bx\,,
	\end{equation}
where the above modified bilinear form $\widetilde{B}(\cdot, \cdot)$ is
\begin{equation}\label{EQ:BILINEAR3}
\begin{aligned}
\widetilde{B}(\delta \Phi, \Psi) &= \sum_{n=1}^{(N+1)/2} \int_{\Omega} \frac{1}{(4n-1)\sigma_{2n-1}} \nabla \delta \varphi_n\cdot \nabla \psi_n d\bx  \\ &\quad + \sum_{n=2}^{(N+1)/2} \int_{\Omega}   (4n-3)\sigma_{2n-2} (\bs_{n}\cdot \delta \Phi) (\bs_{n}\cdot\Psi) d\bx
\\&\quad + \int_{\partial\Omega} \langle R M^{-1} \delta \Phi, \Psi \rangle d\bs\,.
\end{aligned}
\end{equation} 
 Since $RM^{-1}$ is strictly positive definite, the coerciveness of $\widetilde{B}(\cdot, \cdot)$ is immediately deduced from the Poincar\'{e}-Sobolev inequality~\cite{ziemer2012weakly} that $\forall u\in H^1(\Omega)$, $\Omega\subset\bbR^3$, 
\begin{equation}
\|u\|_{L^{6}(\Omega)} \le C_{ps} \left( \|\nabla u\|_{L^2(\Omega)} + \|u\|_{L^1(\partial\Omega)} \right)\,,
\end{equation}
where $C_{ps} = C_{ps}(\Omega)$ is a positive constant depending on $\Omega$ only. Therefore there exists another constant $C_1(N, \Omega)$ that
\begin{equation}\label{EQ:DELTA PHI}
\begin{aligned}
C_1(N, \Omega) \|\delta\Phi\|^2_{[L^2(\Omega)]^{(N+1)/2}} &\le \widetilde{B}(\delta\Phi, \delta\Phi) = -\int_{\Omega} \frac{H_1 - H_2}{\Upsilon} (\bs_1 \cdot \delta\Phi) d\bx \\ &\le \left\| (H_1 - H_2)/\Upsilon\right\|_{L^2(\Omega)} \|\bs_1\cdot  \delta\Phi\|_{L^2(\Omega)} \\
&\le \left\| (H_1 - H_2)/\Upsilon\right\|_{L^2(\Omega)} \left\|\bs_1 \right\|_{\ell^2} \|\delta\Phi\|_{[L^2(\Omega)]^{(N+1)/2}}
\end{aligned}
\end{equation}
by the H\"{o}lder inequality that $$\|\bv\cdot \Phi\|_{W^{k,p}(\Omega)} \le \|\bv\|_{\ell^q} \|\Phi\|_{[W^{k,p}(\Omega)]^{(N+1)/2}},\quad \frac{1}{p}+ \frac{1}{q} = 1,$$ where the norm $\|\cdot\|_{[W^{k,p}(\Omega)]^{(N+1)/2}}$ is defined by
\begin{equation*}
\|\mathbf{f}\|_{[W^{k,p}(\Omega)]^{(N+1)/2}}^p =  \sum_{n=1}^{(N+1)/2} \|f_n\|_{W^{k,p}(\Omega)}^p,\quad \mathbf{f} = [f_1,\dots, f_{(N+1)/2}]^{\text{T}}\,.
\end{equation*}
The estimate~\eqref{EQ:DELTA PHI} implies
\begin{equation}\label{EQ:DELTA PHI BOUND}
\|\delta\Phi\|_{[L^2(\Omega)]^{(N+1)/2}} \le \frac{1}{C_1(N, \Omega)} \left\| (H_1 - H_2)/\Upsilon\right\|_{L^2(\Omega)} \left\|\bs_1 \right\|_{\ell^2}\,.
\end{equation}
Therefore the uniqueness is proved. For the stability estimate, we compute 
\begin{equation}
\begin{aligned}
\frac{H_1 - H_2}{\Upsilon} &= \sigma_{a,1} (\bs_1 \cdot  \Phi) - \sigma_{a,2}(\bs_1 \cdot \widetilde{\Phi}) \\
&= (\sigma_{a,1} -\sigma_{a,2}) ( \bs_1\cdot \Phi)  + \sigma_{a,2} (\bs_1\cdot  \delta\Phi )\,.
\end{aligned}
\end{equation}
Thus using~\eqref{EQ:DELTA PHI BOUND}, we obtain
\begin{equation}\label{EQ: ESTIMATE BOUND}
\begin{aligned}
\|(\sigma_{a,1} -\sigma_{a,2}) (\bs_1\cdot  \Phi)\|_{L^2(\Omega)} &\le \left\|\frac{H_1 - H_2}{\Upsilon}\right\|_{L^2(\Omega)} + \|\sigma_{a,2} (\bs_1 \cdot \delta\Phi)\|_{L^2(\Omega)} \\
&\le \left(1 + \frac{\overline{c}\|\bs_1\|_{\ell^2}^2}{C_1(N, \Omega)}\right)\left\|\frac{H_1 - H_2}{\Upsilon}\right\|_{L^2(\Omega)}\,.
\end{aligned}
\end{equation}
Our proof is completed by noticing~\eqref{EQ:RELATION}.
\end{proof} 
The reconstruction algorithm for $\sigma_a$ is then naturally divided into two steps. First, we solve $\Phi$ from the modified bilinear form~\eqref{EQ:BILINEAR2}, then recover $\sigma_a$ by the relation~\eqref{EQ:RELATION} whenever $\bs_1\cdot \Phi\neq 0$. For the general $SP_N$ system, we cannot guarantee the positivity of $\phi_0=\bs_1\cdot \Phi$ for any positive source function $f(\bx)$.

Under appropriate conditions~\cite{mcclarren2010theoretical}, the $SP_N$ approximation will be eventually converging to the radiative transfer model. However, intuitively, when the order of system $N$ grows, the reconstruction of the coefficients will be less stable due to the coupling of the Legendre moments in the solution. In the following, we study the relation of reconstruction's stability and the system order $N$.  It can be shown that the reconstruction's stability estimate's constant in Theorem~\ref{THM:1} grows at most proportional to $N^{11/8}(1+\log N)$.

\begin{corollary}\label{COR:BOUND}
	Under the assumptions $\cA$-i to $\cA$-iv, suppose $(\Upsilon, \sigma_s)$ are known, $\sigma_{a,1}$,  $\sigma_{a,2}$ are two admissible absorption coefficients, $H_1, H_2$ are the corresponding internal data, respectively. Then $H_1 = H_2$ implies $\sigma_{a,1} = \sigma_{a,2}$ and the following stability estimate holds 
	\begin{equation}
	\begin{aligned}
	\|(\sigma_{a,1} -\sigma_{a,2})\frac{H_1}{\sigma_{a,1}\Upsilon}\|_{L^2(\Omega)} \le C N^{11/8} (1+\log N)  \left\|(H_1 - H_2)/\Upsilon\right\|_{L^2(\Omega)}\,,
	\end{aligned}
	\end{equation}
	where $C = C(\Omega)$ is a positive constant independent of $N$.
\end{corollary}
\begin{proof}
	Recall the estimate~\eqref{EQ: ESTIMATE BOUND}, we only have to give an estimate for $C_1(N, \Omega)$ and $\|\bs_1\|_{\ell^2}^2$ with respect to $N$. Using the equation~\eqref{EQ:BILINEAR3}, we can estimate the lower bound of the coerciveness for $\widetilde{B}(\cdot ,\cdot)$ by neglecting the second term,
	\begin{equation}
	\widetilde{B}(\delta\Phi, \delta\Phi) \ge  \frac{1}{(2N+1)\sigma_s} \|\nabla\delta\Phi\|_{[L^2(\Omega)]^{(N+1)/2}}^2 + \lambda_{(N+1)/2}(RM^{-1}) \|\delta \Phi\|_{[L^2(\partial\Omega)]^{(N+1)/2}}^2\,,
	\end{equation}
	where $\lambda_n(K)$ denotes the $n$-th singular value of $K$ ordered from largest to smallest. Use the inequality introduced in~\cite{marshall1979inequalities}, we estimate the smallest singular value of $RM^{-1}$ that 
	\begin{equation}\label{EQ:LAMBDA}
	\lambda_{(N+1)/2} (R M^{-1}) \ge \lambda_{(N+1)/2}(R) \lambda_{(N+1)/2}(M^{-1}) = \lambda_{(N+1)/2}(R) /\lambda_1(M) \,.
	\end{equation} 
	Since the largest singular value $\lambda_1(M) = \|M\|_{op}$, let $\bv = [v_1, v_2,\dots, v_{(N+1)/2}]^T\in \bbR^{(N+1)/2}$ and take the convention that $v_{(N+3)/2} = 0$, then use the Cauchy-Schwartz inequality, we obtain
	\begin{equation}
	\|M\|_{op}^2 = \sup_{\|\bv\|=1} \|M \bv\|^2 = \sup_{\|\bv\|=1} \sum_{k=1}^{(N+1)/2} ((2k-1)v_{k} + (2k) v_{k+1})^2 \le \note{4N^2 \|\bv\|^2  = 4N^2}\,.
	\end{equation}
	Therefore we have $\lambda_1(M) \le \note{2N}$. In the next,
	we only need to estimate the smallest singular value of $R$. According to the lower bound estimates introduced in~\cite{piazza2002upper,gungor2010erratum}, the smallest singular value satisfies
	\begin{equation}\label{EQ:BOUND}
	\lambda_{(N+1)/2}(R) \ge \left( \frac{ \frac{N-1}{2}}{\|R\|_F^2}\right)^{\frac{N-1}{4}} |\det(R)|\,,
	\end{equation} 
	where $\|\cdot\|_F$ denotes the Frobenius norm, then use the results from Lemma~\ref{LEM:DET} and Lemma~\ref{LEM:FRO} in Appendix, we have the estimate
	\begin{equation}
	\lambda_{(N+1)/2}(R) \ge \left(\frac{N-1}{N+1}\right)^{\frac{N-1}{4}} |\det(R)| = \cO(N^{-3/8})\,.
	\end{equation}
	Therefore $C_1(N, \Omega) \ge c \min(\frac{1}{(2N+1)\sigma_s}, \lambda_{(N+1)/2} (R M^{-1})) = \cO(N^{-11/8})$. To estimate the upper bound of $\|\bs_1\|^2$, \note{we follow the Lemma~\ref{LEM: M INV}} that the row vector $\bs_1 = [s_{1,1},\dots, s_{1, (N+1)/2}]$ is given by
	\begin{equation}
	|s_{1, k}| = \frac{(2k-2)!!}{(2k-1)!!} =  \frac{\sqrt{\pi}}{2} \frac{\Gamma(k)}{\Gamma(k + \frac{1}{2})}\,.
	\end{equation}
	Use the Gautschi's inequality~\cite{gautschi1959some} that
	\begin{equation}
\frac{1}{\sqrt{k+\frac{1}{2}}} 	< \frac{\Gamma(k)}{\Gamma(k + \frac{1}{2})} < \frac{1}{\sqrt{k-\frac{1}{2}}}\,,
	\end{equation}
	we immediately find out
	\begin{equation}
	\|\bs_1\|_{\ell^2}^2 \le \frac{\pi}{4}\sum_{i=1}^{(N+1)/2} \frac{1}{i-\frac{1}{2}} = \cO(1+\log N)\,.
	\end{equation}
	From the result of Theorem~\ref{THM:1}, the stability estimate now can be formulated as
\begin{equation}
\begin{aligned}
\|(\sigma_{a,1} -\sigma_{a,2}) \bs_1\Phi\|_{L^2(\Omega)} \le \cO(N^{11/8}(1+\log N))\left\|\frac{H_1 - H_2}{\Upsilon}\right\|_{L^2(\Omega)}\,.
\end{aligned}
\end{equation}
\end{proof}
\begin{figure}[!htb]
	\centering   
	\includegraphics[scale=0.4]{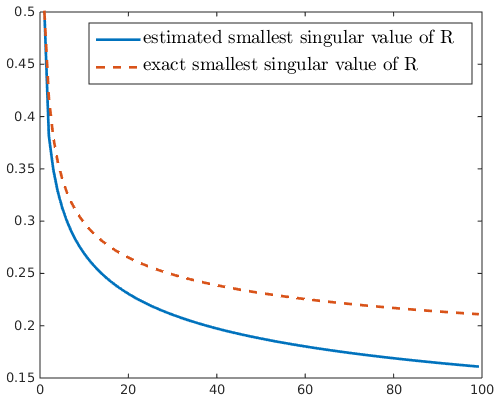}
		\includegraphics[scale=0.4]{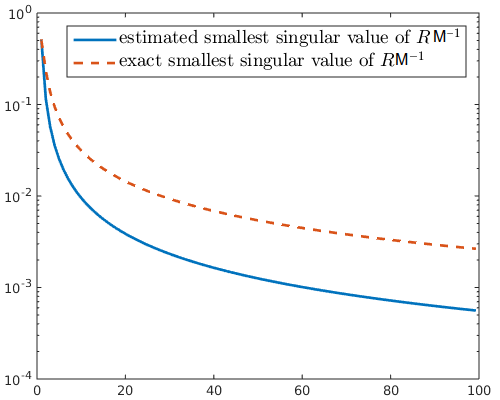}
	\caption{Decay of the smallest singular values with respect to the matrix sizes. The $x$-axis denotes the matrix size $(N+1)/2$. Left: The red (dashed) line represents the exact smallest singular value of $R$ and the blue (solid) line represents the estimated lower bound of the smallest singular value of $R$ through~\eqref{EQ:BOUND}. Right:  The red (dashed) line represents the exact smallest singular value of $RM^{-1}$ and the blue (solid) line represents the estimated lower bound of the smallest singular value of $RM^{-1}$ through~\eqref{EQ:LAMBDA}. }
	\label{FIG:SVD}
\end{figure}
\begin{remark}
	It is possible to improve the above estimate by using a sharper bounded for the Frobenius norm $\|R\|_F$ in Lemma~\ref{LEM:FRO}. The simple bounds~\eqref{EQ:LAMBDA} and~\eqref{EQ:BOUND} are not sharp for the smallest singular value of $R$, see Fig~\ref{FIG:SVD}. It seems possible to achieve better estimate through the calculation of $R^{-1}$'s Frobenius norm by following the technique in~\cite{tyrtyshnikov1992singular}. 
\end{remark}

\begin{remark}
    As $N\to\infty$, the above result shows that the stability estimate's constant will also grow to infinity, this seems to give a negative answer to the uniqueness for qPAT with the radiative transport equation. However, such estimate is only meant for the worst case, since the boundary source could be chosen arbitrarily. In practice, if the source function $f$ is sufficiently smooth, the datum with respect to the $SP_N$ model $H = \Upsilon\sigma_{a}\bs_1\cdot \Phi$ will converge rapidly and the high order modes will decay sufficiently fast, which could counter the growth in the constant. This will be the future work.
\end{remark}
\vspace{0.5cm}
\noindent {\bf{Reconstruction of $\Upsilon$ only.}}
Suppose the coefficients $\sigma_a, \sigma_s$ are known and $\Upsilon$ is unknown, then the $SP_N$ equation system is completely known and $\Phi$ could be uniquely solved, so we can reconstruct $\Upsilon$ explicitly by
\begin{equation}
\Upsilon = \frac{H(\bx)}{\sigma_0(\bx)(\bs_1 \cdot \Phi)}\,.
\end{equation}
The uniqueness and stability estimate will be straightforward, we conclude in the following theorem without proof.
\begin{theorem}
		Under the assumptions $\cA$-i to $\cA$-iv, suppose $(\sigma_a, \sigma_s)$ are known, $\Upsilon_1, \Upsilon_2$ are two Gr\"{u}neisen coefficients, $H_1, H_2$ are the corresponding internal data, respectively. Then we have the following stability estimate
	\begin{equation}
	\|(\Upsilon_1 - \Upsilon_2) \frac{H_1}{\Upsilon_1}\|_{L^2(\Omega)} \le C \|(H_1 - H_2)/\sigma_a\|_{L^2(\Omega)}\,,
	\end{equation}
	the constant $C$ does not depend on $N$.
\end{theorem}

\vspace{0.5cm}
\noindent {\bf{Reconstruction of $\sigma_s$ only.}}
Suppose the coefficients $\Upsilon, \sigma_a$ are known and $\sigma_s$ is unknown, then
\begin{equation}\label{EQ:REDUCTION}
\bs_1  \cdot \Phi = \frac{H(\bx)}{\Upsilon(\bx) \sigma_0(\bx)}
\end{equation}
is known from the measurement $H$. In addition, we also assume that $\sigma_s$ is known on the boundary $\partial\Omega$. When $N=1$, the reconstruction process of $\sigma_s $ will be solving a linear transport equation~\cite{bal2011multi} for ${\sigma_s}^{-1}$, while larger $N$ will introduce extra nonlinearity from the coupling of solution components. The linearized case of $SP_3$ approximation has been recently studied in~\cite{frederick2018image}, which should be able to generalize to $SP_N$ system with the similar technique. In the following, we assume $N \ge 3$, for the corresponding nonlinear inverse problem, let $\bs_k$ be the $k$-th row of $M^{-1}$ and $\bu_{k,n}$ be the $n$-th row of the rank-one matrix $\bs_{k}^T\bs_k$. We reformulate the equations from the bilinear form~\eqref{EQ:BILINEAR} as
\begin{equation}\label{EQ:DIFF}
-\nabla\cdot \left(\frac{1}{\sigma_s} \nabla \varphi_n \right) + \sigma_a \bp_n \cdot \Phi + \sigma_s \bq_n \cdot  \Phi = 0,\quad n=1,\dots, (N+1)/2\,,
\end{equation}
where the row vectors $\bp_n$ and $\bq_n$ are defined by
\begin{equation}\label{EQ: P Q}
\begin{aligned}
\bp_n &= (4n-1)(1 - g^{2n-1}) \bu_{1, n}\,,\\
\bq_n &= (4n- 1)(1 - g^{2n-1}) \sum_{k=2}^{(N+1)/2} (4k-3)(1-g^{2k-2})\bu_{k,n}\,.
\end{aligned}
\end{equation}
For convenience, we also denote $\bP$ and $\bQ$ as the corresponding matrices with $n$-th rows as $\bp_n$ and $\bq_n$, respectively. In the following, we first prove a simple lemma to estimate the variation in solutions with respect to changes in the scattering coefficient.
\begin{lemma}\label{LEM: EST SCATTER}
   Under the assumptions $\cA$-i to $\cA$-iv, suppose $\sigma_{s,1}$ and $\sigma_{s,2}$ are two admissible scattering coefficients that $\delta\sigma_s = \sigma_{s,1} - \sigma_{s,2}$ satisfies $\delta\sigma_s = 0$ on $\partial\Omega$. Let $\Phi_i = [\varphi_{i,1},\dots, \varphi_{i, (N+1)/2}]^T \in [H^1(\Omega)]^{(N+1)/2}$ the solution associated with scattering coefficient $\sigma_{s,i}$, $i=1,2$, then
   \begin{equation}
   \|\Phi_1 - \Phi_2\|_{[H^1(\Omega)]^{(N+1)/2}} \le C \frac{N^{23/8} }{3(1-\note{|g|})} \|\Phi_1\|_{[W^{1,\infty}(\Omega)]^{(N+1)/2}}  \|\sigma_{s,1} - \sigma_{s,2}\|_{L^{2}(\Omega)}\,,
   \end{equation}
   where $C$ is a positive constant independent of $N$.
\end{lemma}
\begin{proof}
    Let $B_i(\cdot, \cdot)$ the bilinear form in~\eqref{EQ:BILINEAR} for coefficient pair $(\sigma_a, \sigma_{s,i})$, $i=1,2$, then for any test function $\Psi\in [H^{1}(\Omega)]^{(N+1)/2}$, we have
    \begin{equation}
    B_1(\Phi_1, \Psi) = B_2(\Phi_2, \Psi)\,.
    \end{equation}
    Denote $\delta\Phi=  [\delta\varphi_{1},\dots, \delta\varphi_{(N+1)/2}]^T := \Phi_1 - \Phi_2$,  by Cauchy-Schwartz inequality, we have
    \begin{equation}\label{EQ:B2U}
    \begin{aligned}
    B_2(\delta\Phi, \delta\Phi) &= \sum_{n=1}^{(N+1)/2} \int_{\Omega} \frac{1}{(4n-1)(1-g^{2n-1})}\left[ \frac{\delta\sigma_s}{\sigma_{s,1}\sigma_{s,2}} \nabla \varphi_{1, n} \cdot \nabla \delta\varphi_n - \delta\sigma_s \bq_n\cdot \Phi_1 \delta\varphi_n\right] d\bx \\
    &\le \theta\sum_{n=1}^{(N+1)/2} \left[ \|\nabla \delta\varphi_n\|_{L^2(\Omega)} \left\| \frac{\delta\sigma_s \nabla \varphi_{1,n}}{\sigma_{s,1}\sigma_{s,2}}\right\|_{L^2(\Omega)} + \frac{3\left\|\delta\varphi_n\right\|_{L^2(\Omega)}}{4n-1}  \left\| \delta\sigma_s \bq_n \cdot \Phi_1 \right\|_{L^2(\Omega)} \right] \\
    &\le \theta\kappa \|\delta\sigma_s\|_{L^{2}(\Omega)} \sum_{n=1}^{(N+1)/2} \|\delta\varphi_n\|_{H^1(\Omega)} \\&\le \theta{\kappa \sqrt{\hN}} \|\delta\sigma_s\|_{L^{2}(\Omega)} \|\delta\Phi\|_{[H^1(\Omega)]^{(N+1)/2}}\,,
     \end{aligned}
    \end{equation}
    where the constants $\theta = (3(1-\note{|g|}))^{-1}$, $\kappa = \sup_{n\ge 1} \left\|\nabla \varphi_{1,n}/(\sigma_{s,1}\sigma_{s,2})\right\|_{L^{\infty}(\Omega)} + \frac{3}{4n-1}\left\| \bq_n\cdot \Phi_1\right\|_{L^{\infty}(\Omega)}$. Since $\sigma_{s,1}, \sigma_{s,2}$ are bounded from below by positive constants, the first term in $\kappa$ is bounded by $\|\Phi_1\|_{[W^{1,\infty}]^{(N+1)/2}}$. The second term needs to estimate $\sup_{n\ge 1}\frac{3}{4n-1}\|\bq_n\|_{\ell^{1}}$. From the definition of $\bq_n$ in~\eqref{EQ: P Q}, we can deduce that
    \begin{equation}
    \begin{aligned}
    \frac{1}{4n-1}\|\bq_n \|_{\ell^{1}} &\le \sum_{k=2}^{(N+1)/2} (4k-3) \|\bu_{k,n}\|_{\ell^{1}} \\
    &=   \sum_{k=2}^{n} (4k-3) \sum_{j=k}^{(N+1)/2} |s_{k, n} s_{k, j}| \\
    &<\sum_{k=2}^{n} \sum_{j=k}^{(N+1)/2} \frac{(k+\frac{1}{2})(4k-3)}{(2k-1)^2} \frac{1}{\sqrt{j-\frac{1}{2}}\sqrt{n-\frac{1}{2}}}\quad \text{(Gaustchi's inequality) }\\
    &\le  \frac{1}{\sqrt{n-\frac{1}{2}}}\sum_{k=2}^n \sum_{j=k}^{(N+1)/2} \frac{2}{\sqrt{j-\frac{1}{2}}} \quad (\text{since } \frac{1}{2}(k+1)(4k-3) \le  (2k-1)^2)\\
    &= \cO(N).
    \end{aligned}
    \end{equation}
    This implies $ \kappa \le \sfc N \|\Phi_1\|_{[W^{1,\infty}(\Omega)]^{(N+1)/2}} $ for certain constant $\sfc > 0$. On the other hand, from the Corollary~\ref{COR:BOUND}, there exists a constant $C_1$ independent of $N$ that
    \begin{equation}\label{EQ:B2L}
    B_2(\delta\Phi, \delta\Phi) \ge C_1 N^{-11/8} \|\delta\Phi\|_{[H^1(\Omega)]^{(N+1)/2}}^2\,.
    \end{equation}
    Combine the estimates~\eqref{EQ:B2L} and~\eqref{EQ:B2U}, we obtain
    \begin{equation}\label{EQ:EST1}
    \|\delta\Phi\|_{H^1(\Omega)} \le \frac{\sfc}{C_1} \frac{N^{23/8} }{3(1-\note{|g|})} \|\Phi_1\|_{[W^{1,\infty}(\Omega)]^{(N+1)/2}}\|\delta\sigma_s\|_{L^{2}}\,.
    \end{equation}
\end{proof}
\begin{theorem}\label{THM: SCATTER}
	Under the assumptions $\cA$-i to $\cA$-iv, suppose $(\sigma_a, \Upsilon)$ are known, let $\sigma_{s,1}$ and $\sigma_{s,2}$ be two admissible scattering coefficients with $\sigma_{s,1} = \sigma_{s,2}$ on the boundary $\partial\Omega$. 	Let $\Phi_1$ and $\Phi_2$ the solutions to the $SP_N$ system with scattering coefficients $\sigma_{s,1}$ and $\sigma_{s,2}$ respectively. $H_1$ and $H_2$ denote the corresponding internal data for $\sigma_{s,1}$ and $\sigma_{s,2}$ respectively. Then  we have the following estimate 
    \begin{equation}
    \int_{\Omega} \cV_N(\bx;\lambda) \left( \frac{\sigma_{s,1} - \sigma_{s,2} }{\sigma_{s,2}}\right)^2 d\bx \le \frac{\sfC}{\lambda} N^2 \left\|\frac{H_1 - H_2}{\Upsilon\sigma_a}\right\|^2_{H^2(\Omega)}\,,
    \end{equation}
    where $\lambda \in (0, 2\underline{c})$ is an arbitrary constant, $\cV_N(\bx;\lambda)$ is
    \begin{equation}
    \begin{aligned}
    \cV_N(\bx) &:= (\sigma_{s,1} + 2\underline{c}-\lambda )(\bs_1\cdot \bQ\Phi_1)^2 + \kappa_N\frac{H_1}{\Upsilon}(\bs_1\cdot\bQ\Phi_1) -\frac{1}{\sigma_{s,1}}\nabla\frac{H_1}{\Upsilon\sigma_a}\cdot\nabla (\bs_1\cdot\bQ\Phi_1)- 2\overline{c}^2\cY\,,
    \end{aligned}
    \end{equation}
    with 
    \begin{equation}
    \begin{aligned}
        \cY &:= C_3 N^{39/8} \|\bs_1\cdot \bQ \Phi_1 \|_{L^{\infty}(\Omega)}\|\Phi_1\|_{W^{1,\infty}(\Omega)}\,,\\
        \kappa_N &:= \sum_{n= 1}^{(N+1)/2} (4n-1)(1 - g^{2n-1}) s_{1,n}^2\,,
    \end{aligned}
    \end{equation}
    the constants $\sfC, C_3$ are independent of $N$. When $\cV_N(\bx;\lambda) > 0$, $\forall \bx \in\Omega$, then $H_1 = H_2$ a.e. on $\Omega$ implies $\sigma_{s,1} = \sigma_{s,2}$ a.e.. 
\end{theorem}
\begin{proof}
For $n=1,\dots, (N+1)/2$, we have the following $SP_N$ systems for $\Phi_1$ and $\Phi_2$,
	\begin{equation}\label{EQ:DIFF2}
	\begin{aligned}
	-\nabla\cdot \left(\frac{1}{\sigma_{s,1}} \nabla \varphi_{1,n} \right) + \sigma_a \bp_n\cdot \Phi_1 + \sigma_{s,1} \bq_n \cdot \Phi_1 &= 0\,, \\
	-\nabla\cdot \left(\frac{1}{\sigma_{s,2}} \nabla \varphi_{2,n} \right) + \sigma_a \bp_n \cdot \Phi_2 + \sigma_{s,2} \bq_n \cdot \Phi_2 &= 0\,,
	\end{aligned}
	\end{equation}
    with the mixed boundary condition~\eqref{EQ:BD}.
	Since $\Upsilon$ and $\sigma_a$ are known, the measurements are given by
	\begin{equation}
	\bs_1 \cdot \Phi_i = \frac                                                                                                                                                               {H_i}{\Upsilon\sigma_a},\quad i=1,2\,.
	\end{equation}
	Therefore multiply~\eqref{EQ:DIFF2} with $s_{1,n}$ and take summation over $n$, we get
	\begin{equation}\label{EQ:DIFF3}
	-\nabla\cdot \left(\frac{1}{\sigma_{s, i}} \nabla \frac{H_i}{\Upsilon\sigma_a}\right) +\bs_1\cdot \left( \sigma_a \bP \Phi_i + \sigma_{s,i} \bQ \Phi_i\right) = 0,\quad i =1,2  \,.
	\end{equation}
	Take the difference between equations~\eqref{EQ:DIFF3} with $\sigma_{s,1}$ and $\sigma_{s,2}$, respectively. Let $\delta\sigma_s = \sigma_{s,1} - \sigma_{s,2}$, $\delta\Phi = \Phi_1 - \Phi_2$ and $\delta H  = H_1 - H_2$, then
	\begin{equation}\label{EQ:LINEARIZE}
	-\nabla\cdot\left(\frac{-\delta\sigma_s}{\sigma_{s,1}\sigma_{s,2}}\nabla \frac{H_1}{\Upsilon\sigma_a}\right)- 	\nabla\cdot \left(\frac{1}{\sigma_{s, 2}} \nabla \frac{\delta H}{\Upsilon\sigma_a}\right) +\bs_1 \cdot \left(\sigma_a \bP  \delta\Phi + \delta\sigma_s \bQ \Phi_1 + \sigma_{s,2}\bQ \delta\Phi\right) = 0\,.
	\end{equation}
	Use the following identity, 
	\begin{equation}
	 \frac{\delta\sigma_s}{\sigma_{s,2}} \nabla \cdot \left(\frac{\delta\sigma_s}{\sigma_{s,2}} \frac{1}{\sigma_{s,1}} \nabla \frac{H_1}{\Upsilon\sigma_a} \right) = \frac{1}{2}\left(\frac{\delta\sigma_s}{\sigma_{s,2}}\right)^2 \nabla \cdot \left(\frac{1}{\sigma_{s,1}} \nabla \frac{H_1}{\Upsilon\sigma_a}\right) + \frac{1}{2}\nabla \cdot \left(\left(\frac{\delta \sigma_s}{\sigma_{s,2}}\right)^2 \frac{1}{\sigma_{s,1}}\nabla \frac{H_1}{\Upsilon\sigma_a}\right)\,,
	\end{equation}
	we multiply~\eqref{EQ:LINEARIZE} with $\delta \sigma_s/\sigma_{s,2}$, then
	\begin{equation}\label{EQ:SIGMA1}
	\begin{aligned}
&\frac{1}{2}\left(\frac{\delta\sigma_s}{\sigma_{s,2}}\right)^2 \nabla \cdot \left(\frac{1}{\sigma_{s,1}} \nabla \frac{H_1}{\Upsilon\sigma_a}\right)+ \frac{1}{2}\nabla \cdot \left(\left(\frac{\delta \sigma_s}{\sigma_{s,2}}\right)^2 \frac{1}{\sigma_{s,1}}\nabla \frac{H_1}{\Upsilon\sigma_a}\right)\\& - \frac{\delta \sigma_s}{\sigma_{s,2}}	\nabla\cdot \left(\frac{1}{\sigma_{s, 2}} \nabla \frac{\delta H}{\Upsilon\sigma_a}\right) + \frac{\delta \sigma_s}{\sigma_{s,2}}\bs_1\cdot \left(\sigma_a \bP \delta\Phi + \delta\sigma_s \bQ \Phi_1 + \sigma_{s,2}\bQ \delta\Phi\right) = 0\,.
	\end{aligned}
	\end{equation}
The first term can be replaced from~\eqref{EQ:DIFF3} that
\begin{equation}\label{EQ:SIGMA2}
\frac{1}{2}\left(\frac{\delta\sigma_s}{\sigma_{s,2}}\right)^2 \nabla \cdot \left(\frac{1}{\sigma_{s,1}} \nabla \frac{H_1}{\Upsilon\sigma_a}\right) = \frac{1}{2}\left(\frac{\delta\sigma_s}{\sigma_{s,2}}\right)^2 \bs_1\cdot \left( \sigma_a \bP \Phi_1 + \sigma_{s,1} \bQ \Phi_1\right)\,,
\end{equation}
then combine~\eqref{EQ:SIGMA1} and~\eqref{EQ:SIGMA2}, multiply the test function $\psi=\bs_1\cdot \bQ\Phi_1\in H^1(\Omega)$ to~\eqref{EQ:SIGMA1} and integrate over $\Omega$. Notice that $\delta\sigma_s = 0$ on $\partial\Omega$, we obtain
\begin{equation}\label{EQ:KEY}
\begin{aligned}
&\frac{1}{2}\int_{\Omega} \left(\frac{\delta\sigma_s}{\sigma_{s,2}}\right)^2\left[ \bs_1\cdot \left( \sigma_a\bP  + \sigma_{s,1} \bQ  + 2\sigma_{s,2} \bQ \right)\Phi_1(\bs_1\cdot \bQ\Phi_1)-\frac{1}{\sigma_{s,1}}\nabla\frac{H_1}{\Upsilon\sigma_a}\cdot\nabla (\bs_1\cdot\bQ\Phi_1)\right]  d\bx \\
&- \int_{\Omega}  \frac{\delta \sigma_{s}}{\sigma_{s,2}}\left[ \nabla\cdot \left( \frac{1}{\sigma_{s,2}}\nabla \frac{\delta H}{\Upsilon\sigma_a}\right)   - \bs_1 \cdot(\sigma_a\bP \delta\Phi  + \sigma_{s,2} \bQ \delta \Phi) \right] (\bs_1\cdot \bQ\Phi_1) d\bx= 0\,.
\end{aligned}
\end{equation}
In the next, observe that
\begin{equation}
\begin{aligned}
\bs_1\cdot \bP = \sum_{n= 1}^{(N+1)/2} (4n-1)(1 - g^{2n-1}) s_{1,n} \bu_{1,n}  = \left(\sum_{n= 1}^{(N+1)/2}  (4n-1)(1 - g^{2n-1}) s_{1,n}^2\right) \bs_{1} \,,
\end{aligned}
\end{equation}
therefore in~\eqref{EQ:KEY}, we can replace
\begin{equation}
\begin{aligned}
\sigma_a \bs_1\cdot \bP\delta\Phi = \left(\sum_{n= 1}^{(N+1)/2} (4n-1)(1 - g^{2n-1}) s_{1,n}^2\right) \frac{\delta H}{\Upsilon}\,, \\\sigma_a \bs_1\cdot \bP \Phi_1 = \left(\sum_{n= 1}^{(N+1)/2} (4n-1)(1 - g^{2n-1}) s_{1,n}^2\right) \frac{H_1}{\Upsilon} \,.
\end{aligned}
\end{equation}
Define the constant $\kappa_N := \sum_{n = 1}^{(N+1)/2} (4n-1)(1 - g^{2n-1}) s_{1,n}^2$, since $|s_{1,n}|^2 = \Theta(n^{-1})$, then $\kappa_N = \Theta(N)$, the equation~\eqref{EQ:KEY} can be further reduced to
\begin{equation}\label{EQ:KEY2}
\begin{aligned}
&\frac{1}{2}\int_{\Omega} \left(\frac{\delta\sigma_s}{\sigma_{s,2}}\right)^2\left[\left(\sigma_{s,1}  + 2\sigma_{s,2}  \right) (\bs_1\cdot \bQ\Phi_1)^2 + \kappa_N \frac{H_1}{\Upsilon} (\bs_1\cdot \bQ\Phi_1)-\frac{1}{\sigma_{s,1}}\nabla\frac{H_1}{\Upsilon\sigma_a}\cdot\nabla (\bs_1\cdot \bQ\Phi_1)\right]  d\bx\\
&= \int_{\Omega} \frac{\delta \sigma_s}{\sigma_{s,2}}\left[ \nabla\cdot \left( \frac{1}{\sigma_{s,2}}\nabla \frac{\delta H}{\Upsilon\sigma_a}\right) - \kappa_N \frac{\delta H}{\Upsilon}   -  \sigma_{s,2} \left(\bs_1\cdot \bQ \delta \Phi\right) \right](\bs_1\cdot \bQ\Phi_1)d\bx\,.
\end{aligned}
\end{equation}
Due to Lemma~\ref{LEM: SQ BD}, $\|\bs_1\cdot \bQ\|_{\ell^{p}}\le C_2 N^{3/2+1/p} $ for certain constant $C_2$ independent of $N$, then combine with Lemma~\ref{LEM: EST SCATTER},  the second term on right-hand-side of~\eqref{EQ:KEY2} is bounded by 
\begin{equation}\nonumber
\begin{aligned}
\big|\int_{\Omega} \delta \sigma_{s} \left(\bs_1\cdot \bQ \delta \Phi\right)(\bs_1\cdot \bQ\Phi_1) d\bx \big| &\le \|\bs_1\cdot \bQ\|_{\ell^2}  \left\|\delta\sigma_s\right\|_{L^2(\Omega)} \|\delta\Phi\|_{L^2(\Omega)}\|\bs_1\cdot \bQ \Phi_1\|_{L^{\infty}(\Omega)}\\
&\le C_2 N^{2} \left\|\delta\sigma_s\right\|_{L^2(\Omega)} \|\delta\Phi\|_{L^2(\Omega)} \|\bs_1\cdot \bQ\Phi_1\|_{L^{\infty}(\Omega)}\\
&\le C_3 N^{39/8} \|\Phi_1\|_{W^{1,\infty}(\Omega)} \|\bs_1\cdot\bQ\Phi_1\|_{L^{\infty}(\Omega)} \left\|\delta\sigma_s\right\|_{L^2(\Omega)}^2\,,
\end{aligned}
\end{equation}
where the constant $C_3$ is independent of $N$ as well. Let $\cX$ and $\cY$ denote the following quantities
\begin{equation}
\begin{aligned}
\cX&:=  \kappa_N\frac{H_1}{\Upsilon}(\bs_1\cdot\bQ\Phi_1) -\frac{1}{\sigma_{s,1}}\nabla\frac{H_1}{\Upsilon\sigma_a}\cdot\nabla (\bs_1\cdot\bQ\Phi_1)\,, \\
\cY &:= C_3 N^{39/8} \|\Phi_1\|_{W^{1,\infty}(\Omega)} \|\bs_1\cdot \bQ \Phi_1\|_{L^{\infty}(\Omega)} \,,
\end{aligned}
\end{equation}
then we obtain the following inequality,
\begin{equation}
\begin{aligned}
&\frac{1}{2}\int_{\Omega} \left(\frac{\delta\sigma_s}{\sigma_{s,2}}\right)^2 \left[(\sigma_{s,1} + 2\sigma_{s,2})(\bs_1\cdot \bQ\Phi_1)^2 + \cX- 2\sigma_{s,2}^2\cY\right] d\bx \\
&\le  \int_{\Omega} \frac{\delta \sigma_s}{\sigma_{s,2}}\left( \nabla\cdot \left( \frac{1}{\sigma_{s,2}}\nabla \frac{\delta H}{\Upsilon\sigma_a}\right) - \kappa_N \frac{\delta H}{\Upsilon} \right) (\bs_1\cdot \bQ\Phi_1) d\bx \\
&\le \frac{\lambda}{2}\int_{\Omega} \left[\frac{\delta \sigma_s}{\sigma_{s,2}}(\bs_1\cdot\bQ\Phi_1)\right]^2 d\bx + \frac{1}{ 2\lambda}\int_{\Omega}\left( \nabla\cdot \left( \frac{1}{\sigma_{s,2}}\nabla \frac{\delta H}{\Upsilon\sigma_a}\right) - \kappa_N \frac{\delta H}{\Upsilon} \right)^2d\bx\,,
\end{aligned}
\end{equation}
where $\lambda > 0$ is an arbitrary number and the last inequality has used the AM-GM inequality. For the uniqueness, we let $\delta H = 0$ in above inequality, then it becomes
\begin{equation}
\int_{\Omega} \left(\frac{\delta\sigma_s}{\sigma_{s,2}}\right)^2 \left[(\sigma_{s,1} + 2\sigma_{s,2} -\lambda )(\bs_1\cdot \bQ\Phi_1)^2 + \cX -2 \sigma_{s,2}^2\cY \right] d\bx  \le 0\,.
\end{equation}
Recall that $\underline{c}\le \sigma_{s,2}\le \overline{c}$, therefore if
\begin{equation}
(\sigma_{s,1} + 2\underline{c}-\lambda )(\bs_1\cdot \bQ\Phi_1)^2 + \kappa_N\frac{H_1}{\Upsilon}(\bs_1\cdot\bQ\Phi_1) -\frac{1}{\sigma_{s,1}}\nabla\frac{H_1}{\Upsilon\sigma_a}\cdot\nabla (\bs_1\cdot\bQ\Phi_1) > 2\overline{c}^2\cY\,,
\end{equation}
we could conclude that $\delta\sigma_s = 0$ a.e.. The stability estimate is straightforward by noticing $\kappa_N = \cO(N)$.
\end{proof}
\begin{remark}
    As $N\to\infty$, the requirement that $\cV(x;\lambda) > 0$ could be difficult to fulfill since the growth of $\cY$ is much faster than the other terms. The estimate could be greatly improved by giving a tighter bound to $\int_{\Omega} \delta\sigma_s (\bs_{1}\cdot \bQ\delta\Phi)(\bs_1\cdot \bQ\Phi_1) d\bx$, for instance,  estimate the Frech\'et derivative of $\frac{\delta\Phi}{\delta\sigma_s}$.
\end{remark}
In the next, we focus on two important cases of simultaneous reconstructions: $(\Upsilon, \sigma_a)$ and $(\sigma_a, \sigma_s)$ with multiple illumination sources. In $SP_1$      , one can only reconstruct any two coefficients with the knowledge about the third one~\cite{bal2011multi,ren2013hybrid} and it is impossible to recover all of them unless additional information is provided. In $SP_N$, it is still unclear whether or not all of the coefficients can be recovered uniquely. 

\vspace{0.5cm}
\noindent {\bf{Reconstruction of $\sigma_a$ and $\Upsilon$.}}
In this case, we consider the simultaneous reconstruction of both $\sigma_a$ and $\Upsilon$ with multiple sources $f_j$, $1\le j\le J$ ($J\ge 2$). We denote $H_j$ the corresponding measurement for $f_j$ from qPAT. The simplest case $SP_1$ is studied in~\cite{bal2011multi,ren2013hybrid} and linearized case of $SP_3$ is discussed in~\cite{ frederick2018image}. The key observation is that the ratio of two measurements is independent of the coefficients, which only implicitly depends on $\sigma_a$. 
We first consider the linearized setting, suppose the scattering coefficient $\sigma_s$  and the background absorption and Gr\"{u}neisen coefficients $\sigma_{a}, \Upsilon$ are known and the perturbations are $\delta\note{\sigma_a}, \delta\Upsilon$. For each $1\le j\le J$, suppose $\Phi_j$ is the background solution with source $f_j$, the perturbation in the solution is denoted by $\delta\Phi_j$, then by linearizing~\eqref{EQ:DIFF}, $\delta\Phi_j = [\delta\varphi_{j,1},\dots, 
\delta\varphi_{(N+1)/2}]^{\text{T}}$ satisfies the linearized $SP_N$ system
\begin{equation}\label{EQ: DELTA PHI}
-\nabla\cdot \left(\frac{1}{\sigma_s} \nabla \delta\varphi_{j,n} \right) + \sigma_a \bp_n \cdot \delta\Phi_j + \sigma_s \bq_n \cdot  \delta\Phi_j = -\delta \sigma_a \bp_n \cdot \Phi_j,\quad n=1,\dots, (N+1)/2\,.
\end{equation}
For any pair of indices $1\le i <  j\le J$, we define the following quantity 
\begin{equation}
\cH_{ij} := (\bs_1\cdot \Phi_i)\frac{\delta H_j}{\sigma_a\Upsilon} - (\bs_1\cdot \Phi_j)\frac{\delta H_i}{\sigma_a\Upsilon} = (\bs_1\cdot \Phi_i)(\bs_1\cdot \delta\Phi_j) - (\bs_1\cdot \Phi_j)(\bs_1\cdot \delta\Phi_i)\,,
\end{equation}
which is known and only depends on $\delta \sigma_a$ and independent of $\Upsilon$. Our reconstruction will be a natural two-step process, first solve $\delta\sigma_a$ from the crossing quantity $\cH_{ij}$ ($1\le i < j \le J$), then solve $\delta\Phi_i$ using the recovered $\delta\sigma_a$, finally if $\sigma_a > 0$ find $\delta\Upsilon$ through
\begin{equation}
\delta\Upsilon =  \frac{1}{J}\sum_{j = 1}^J\frac{\delta H_j - \Upsilon\sigma_a \bs_1\cdot \delta\Phi_j- \Upsilon\delta\sigma_a\bs_1\cdot \Phi_j}{\sigma_a\bs_1\cdot \Phi_j}\,.
\end{equation}
By taking linearization over the bilinear form~\eqref{EQ:BILINEAR}, for any test function $\Psi\in [H^1(\Omega)]^{(N+1)/2}$,
    \begin{equation}
    B(\delta\Phi_j, \Psi) = -\int_{\Omega} \left[\delta\sigma_a(\bs_1\cdot \Phi_j)\bs_1  \right]\cdot \Psi d\bx\,.
    \end{equation}
    Take $\Psi = \delta\Phi_j$ \note{and use the fact 
    \begin{equation}
        B(\delta\Phi_j, \delta\Phi_j) = \tilde{B}(\delta\Phi_j, \delta\Phi_j) + \int_{\Omega} \sigma_a (\bs_1\cdot \delta\Phi_j)^2 d\bx \ge  \tilde{B}(\delta\Phi_j, \delta\Phi_j),
    \end{equation}
    } we can easily conclude the following estimates from the coerciveness of \note{$\tilde{B}(\cdot, \cdot)$},
    \begin{equation}
    \begin{aligned}
    \|\delta\Phi_j\|_{[H^1(\Omega)]^{(N+1)/2}}& =\cO\left( N^{11/8}\|\delta \sigma_a(\bs_1\cdot \Phi_j)\|_{L^2(\Omega)} \|\bs_1\|_{\ell^2}\right)\\&= \cO\left( N^{11/8}(1+\log N)\|\delta \sigma_a(\bs_1\cdot \Phi_j)\|_{L^2(\Omega)}\right)\,,
    \end{aligned}
    \end{equation}
   On the other hand, multiply~\eqref{EQ: DELTA PHI} with $s_{1,n}$ and sum over $n$,
    \begin{equation}
    \delta\sigma_a \kappa_N (\bs_{1}\cdot \Phi_j) = -\nabla\cdot \left(\frac{1}{\sigma_s}\nabla (\bs_{1}\cdot \delta\Phi_j)\right) +\sigma_a \kappa_N \bs_1\cdot \delta\Phi_j + \sigma_{s} \bs_1\cdot \bQ \delta\Phi_j\,,
    \end{equation}
    therefore we have a straightforward estimate
    \begin{equation}
    \|\delta\sigma_a (\bs_1\cdot \Phi_j)\|_{L^2(\Omega)} =\cO\left( \frac{\max(\kappa_N\|\bs_1\|_{\ell^2}, \|\bs_1\cdot \bQ\|_{\ell^2})}{\kappa_N} \|\delta\Phi_j\|_{[H^2(\Omega)]^{(N+1)/2}}\right) \,,
    \end{equation}
    where the constant $\kappa_N = \sum_{n = 1}^{(N+1)/2} (4n-1)(1 - g^{2n-1}) s_{1,n}^2 = \Theta(N)$ and $\|\bs_1\cdot \bQ\|_{\ell^2} \le \cO(N^2)$ from Lemma~\ref{LEM: SQ BD}, \note{where $\Theta$ denotes the Laudau big Theta notation}. These two estimates imply that
    \begin{equation}
    c^{-1}N^{-11/8}(1 + \log N)^{-1} \|\delta \Phi_j\|_{[H_1(\Omega)]^{(N+1)/2}}\le \|\delta\sigma_a (\bs_1\cdot \Phi_j)\|_{L^2(\Omega)} \le  cN\|\delta \Phi_j\|_{[H^2(\Omega)]^{(N+1)/2}}.
    \end{equation}
    for some constant $c>1$ independent of $N$. 
    
    Therefore if there exists two source distinct functions $f_i, f_j$ such that the linear mapping $\delta\sigma_a \mapsto \cH_{ij}$ is invertible and $\bs_1\cdot \Phi_j\neq 0$ over $\Omega$, then one can recover both $\delta\sigma_a$ and $\delta\Phi_j$ from $\cH_{ij}$ uniquely. In general such problem is ill-posed due to the compactness of the mapping $\delta\sigma_a\mapsto \cH_{ij}$, numerical reconstruction of $\delta\sigma_a$ can be done through the following minimization formulation with regularization,
    \begin{equation}\nonumber
    \delta\sigma_a^{\ast} = \argmin_{\delta\sigma_a} \sum_{1\le i < j \le J}\left[  \Big\|\cH_{ij} - \left[(\bs_1\cdot \Phi_i)(\bs_1\cdot \delta\Phi_j) - (\bs_1\cdot \Phi_j)(\bs_1\cdot \delta\Phi_i) \right]\Big\|_{L^2(\Omega)}^2\right] +\alpha \Big\|\nabla \delta\sigma_a\Big\|_{L^2(\Omega)}^2\,,
    \end{equation}
    where $\alpha$ is the regularization parameter.
    
     Particularly, when the background absorption coefficient $\sigma_a = 0$ or negligible, then we approximately have $\delta H_i = \delta\sigma_a \Upsilon \bs_1\cdot \Phi_i$, which does not contain the perturbation $\delta\Upsilon$, in this case, we can only reconstruct $\delta\sigma_a$, the stability estimate is similar to the Theorem~\ref{THM:1}. 
     
     Without linearization, we take the ratio of two data sets $H_i$ and $H_j$, then 
     \begin{equation}
     \frac{H_i}{H_j} = \frac{\bs_1\cdot \Phi_i}{\bs_1\cdot \Phi_j}\,,
     \end{equation}
     which only depends on $\sigma_a$, therefore our reconstruction strategy is similar to the linearized case. First, try to solve the minimization problem:
     \begin{equation}
     \sigma_a^{\ast} = \argmin_{\sigma_a} \sum_{1\le i < j \le J} \|H_i(\bs_1\cdot \Phi_j) - H_j(\bs_1\cdot \Phi_i)\|^2_{L^2(\Omega)} + \alpha \|\nabla \sigma_a\|_{L^2(\Omega)}^2\,.
     \end{equation}
     Then compute $\Upsilon^{\ast} =\frac{1}{J} \sum_{1\le i\le J}\frac{H_i}{\sigma_a^{\ast} \bs_1\cdot\Phi_i}$ with the reconstructed $\sigma_a^{\ast}$.
     
\vspace{0.5cm}
\noindent {\bf{Reconstruction of $\sigma_s$ and $\sigma_a$.}}
We consider the simultaneous reconstruction of both $\sigma_s$ and $\sigma_a$ from multiple sources $f_j$, $1\le j\le J$ provided that $\Upsilon$ is known. Similar to the previous case, we denote $H_j$ the measurement for $f_j$ from the qPAT experiments. Under the linearized setting, let $\sigma_a$ and $\sigma_s$ the background absorption and scattering coefficients, the corresponding perturbations are $\delta\sigma_a$ and $\delta\sigma_s$. For each source $f_j$, let $\Phi_j$ the background solution and the perturbation in the solution is $\delta\Phi_j$, the corresponding perturbation in the measurement is $\delta H_j$. Linearize the variational form~\eqref{EQ:BILINEAR2},  we obtain the following equation,
\begin{equation}\nonumber
\begin{aligned}
\widetilde{B}(\delta\Phi_j, \Psi) = -\int_{\Omega} \frac{\delta H_j}{\Upsilon}(\bs_1\cdot \Psi) d\bx +\sum_{n=1}^{(N+1)/2} \int_{\Omega}\frac{\delta\sigma_s}{(4n-1)(1-g^{2n-1})}\left[ \frac{1}{\sigma_{s}^2} \nabla \varphi_{1, n} \cdot \nabla \psi_n -  \bq_n\cdot \Phi_j \psi_n\right] d\bx\,,
\end{aligned}
\end{equation}
where the bilinear form $\widetilde{B}(\cdot,\cdot)$ is from~\eqref{EQ:BILINEAR3} and $\bq_n$ is defined in~\eqref{EQ: P Q}. Hence the perturbation $\delta\Phi_j$ only linearly depends on $\delta\sigma_s$. On the other hand, since $\delta H_j/ \Upsilon = \delta\sigma_a \bs_1\cdot \Phi_j + \sigma_a \bs_1\cdot\delta\Phi_j$, the crossing quantity $\cH_{ij} = \frac{1}{\sigma_a\Upsilon^2}(H_i\delta H_j - H_j \delta H_i) = (\bs_1\cdot \Phi_i)(\bs_1\cdot \delta\Phi_j) - (\bs_1\cdot \Phi_j)(\bs_1\cdot\delta\Phi_i)$ only linearly depends on $\delta\sigma_s$, therefore we first try to reconstruct $\delta\sigma_s$ from $\cH_{ij}$, then find $\delta\Phi_j$ and recover $\delta\sigma_a$ using 
\begin{equation}
\delta\sigma_a = \frac{1}{J}\sum_{j=1}^J\left[ \left(\frac{\delta H_j}{\Upsilon} - \sigma_a \bs_1\cdot \delta\Phi_j\right)/(\bs_1\cdot\Phi_j)\right]\,.
\end{equation}
Similar to the previous case, the uniqueness is immediate if the crossing term $\cH_{ij}$ as a linear functional of $\delta\sigma_s$ is uniquely solvable and $\bs_1\cdot \Phi_j\neq 0$ over $\Omega$. However since $\delta\sigma_s\mapsto \delta\Phi_j$ is a compact mapping, the inverse problem is ill-posed. Numerically, we consider the following $L^2$ optimization formulation with regularization:
 \begin{equation}\nonumber
 \delta\sigma_s^{\ast} = \argmin_{\delta\sigma_s}\sum_{1\le i < j \le J} \left[\Big\|\cH_{ij} - \left[(\bs_1\cdot\Phi_i)(\bs_1\cdot \delta\Phi_j) - (\bs_1\cdot\Phi_j)(\bs_1\cdot\delta\Phi_i) \right]\Big\|_{L^2(\Omega)}^2\right]  +\alpha \Big\|\nabla \delta\sigma_s\Big\|_{L^2(\Omega)}^2\,.
 \end{equation}

Additionally, if we are provided \emph{a priori} estimate on the perturbation $\delta\sigma_a$  that  $\|\frac{\delta\sigma_a H_i}{\Upsilon\sigma_a^2}\|_{H^2(\Omega)}\le \eps\ll 1$ for certain $1\le i\le J$, then linearize the equation~\eqref{EQ:DIFF3} for $\Phi_i$, we obtain
\begin{equation}\label{EQ:LINEARIZE2}
-\nabla\cdot\left(\frac{-\delta\sigma_s}{\sigma_{s}^2}\nabla \bs_1\cdot\Phi_i\right)- 	\nabla\cdot \left(\frac{1}{\sigma_{s}} \nabla \bs_1\cdot\delta\Phi_i\right) +\kappa_N\frac{\delta H_i}{\Upsilon}+\bs_1\cdot \left(\delta\sigma_s \bQ \Phi_i + \sigma_{s}\bQ \delta\Phi_i\right) = 0\,,
\end{equation}
where $\bQ$ is defined in~\eqref{EQ: P Q} and $\kappa_N =\sum_{n = 1}^{(N+1)/2} (4n-1)(1 - g^{2n-1}) s_{1,n}^2$. Following the similar approach in Theorem~\ref{THM: SCATTER}, 
\begin{equation}\label{EQ:KEY2 CASE5}
\begin{aligned}
&\frac{1}{2}\int_{\Omega} \left(\frac{\delta\sigma_s}{\sigma_{s}}\right)^2\left[\kappa_N\frac{H_i}{\Upsilon}(\bs_1\cdot\bQ\Phi_i) + 3\sigma_{s} (\bs_1\cdot\bQ\Phi_i)^2-\frac{1}{\sigma_{s}}\nabla(\bs_1\cdot\Phi_i)\cdot\nabla (\bs_1\cdot\bQ\Phi_i)\right]  d\bx\\
&= \int_{\Omega} \frac{\delta \sigma_s}{\sigma_{s}}\left[ \nabla\cdot \left( \frac{1}{\sigma_{s}}\nabla \bs_1\cdot\delta\Phi_i\right) - \kappa_N \frac{\delta H_i}{\Upsilon}   -  \sigma_{s} \left(\bs_1\cdot\bQ \delta \Phi_i\right) \right](\bs_1\cdot\bQ\Phi_i)d\bx\,,
\end{aligned}
\end{equation}
Replace $\bs_1\cdot\delta\Phi_i = (\delta H_i - \delta\sigma_a H_i/\sigma_a) / (\Upsilon \sigma_a)$ and 
we immediately get the following estimate from the argument of Theorem~\ref{THM: SCATTER} that
\begin{equation}
\begin{aligned}
\int_{\Omega} \cV_N(\bx) \left( \frac{\delta\sigma_{s}}{\sigma_{s}}\right)^2 d\bx &\le \sfC \left(N^2 \left\|\frac{\delta H_i}{\Upsilon\sigma_a}\right\|^2_{H^2(\Omega)} + \left\|\frac{H_i\delta \sigma_a}{\Upsilon\sigma_a^2}\right\|^2_{H^2(\Omega)}\right) \\
&\le \sfC \left(N^2 \left\|\frac{\delta H_i}{\Upsilon\sigma_a}\right\|^2_{H^2(\Omega)} +\eps^2\right)\,,
\end{aligned}
\end{equation}
where $\cV_N(\bx)$ is the same as in the Theorem~\ref{THM: SCATTER}, $\sfC$ is a constant independent of $N$. 
\section{Numerical experiments}
\label{SEC:Num}
In this section, we perform our numerical experiments in two phases: (i). Assume the true model is certain $SP_N$ system and then reconstruct the coefficients with exactly the same model; (ii). Assume the true model is either radiative transport equation or certain high order $SP_N$ system, then the reconstruction is performed over a low order $SP_N$ system. 

In all the following numerical experiments, we use the unit square in 2D as our domain $\Omega$. It is worthwhile to notice that all the previous arguments are meant for 3D only, the 2D experiments here should be interpreted as special cases (e.g. infinite tube) of 3D, assuming the solution is independent of the third dimension. \note{If the true model is the $SP_N$ system,} the forward problem is solved though finite element method on a sufficiently fine mesh and the inverse problem is solved on a different mesh to avoid inverse crime. \note{If the true model is the radiative transport equation, there are many fast forward solvers available~\cite{ren2019fast, ren2019separability, fan2019fast, gao2013analysis}, we select the finite element method implementation mentioned in~\cite{li2020inverse} for convenience purpose.} The source code for the numerical experiments is hosted on GitHub\footnote{\href{https://github.com/lowrank/spn_qpat}{https://github.com/lowrank/spn\_qpat}}.

\subsection{Experiment setting}
\label{SEC: EXP SET}
In the following numerical experiments, we will use the boundary source functions $f_1(x, y) = 1+ x$, $f_2(x,y) = 1 + \sin(4 \pi x)$ and the anisotropy constant $g = 0.8$. The coefficients $(\sigma_a, \sigma_s, \Upsilon)$ are selected from the following variable set, see Fig~\ref{FIG: COEF SET 1}. 
\begin{figure}[!htb]
    \centering
    \includegraphics[scale=0.393]{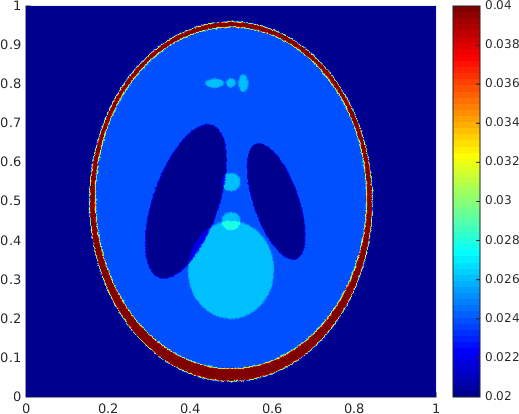}
    \includegraphics[scale=0.393]{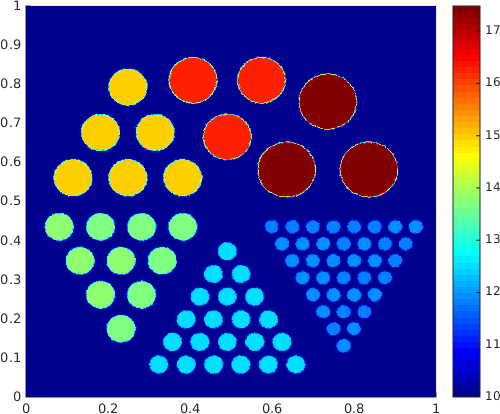}
    \includegraphics[scale=0.393]{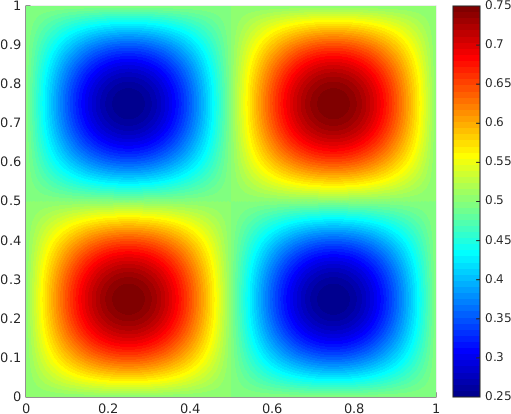}
    \caption{Coefficient Set. From left to right: absorption coefficient $\sigma_{a}$ (Shepp-Logan phantom), scattering coefficient $\sigma_{s}$ (Jaszczak phantom), Gr\"uneisen coefficient $\Upsilon$.}
    \label{FIG: COEF SET 1}
\end{figure}

\subsection{Validation of approximation}
\label{SEC: VALID APPOX}
Before we start to run any of the numerical experiments, we need to verify that if our $SP_N$ model is valid under these settings,  which means modeling error should not be dominating (in practice there are noises in data). Therefore it is important to compare the quantity $\phi_0^N = \bs_1\cdot \Phi = \frac{H}{\sigma_a \Upsilon}$ with the solution's angular average $U(\bx) = \int_{\bbS^{d-1}} u(\bx, \bv) d \bv$ for the radiative transport equation~\eqref{EQ:RTE}. We summarize the $L^2$ relative errors: $\frac{\|\phi_0(\bx) - U(\bx)\|}{\|U(\bx)\|}$ with respect to different $SP_N$ models in Tab~\ref{TAB: VALID}.

\begin{table}[!htb]
    \centering
    \caption{Relative $L^2$ error of the $\phi_0^N$ from $SP_N$ equations with the angular averaged solution $U(\bx)$ from radiative transport equation for boundary sources $f_1$ and $f_2$.}
    \label{TAB: VALID}
    \begin{tabular}{||c|c|c|c|c|c|c|c|c|c||}
        \hline
       & $SP_1$  & $SP_3$ &  $SP_5$ & $SP_7$ &   $SP_9$ & $SP_{11}$ & $SP_{13}$ & $SP_{15}$ & $SP_{17}$  \\
        \hline
      $f_1$ & {\bf{1.93\%}}& 2.28\%& 2.25\%& 2.24\%& 2.23\%& 2.23\%& 2.23\%& 2.23\% & 2.23\%\\
      \hline
      $f_2$  &5.24\%& 3.98\%& {{3.76\%}}&{\bf{3.75\%}}& 3.77\%& {{3.78\%}}& 3.79\%& 3.79\% & 3.79\%\\
        \hline
    \end{tabular}
\end{table}


From the table Tab~\ref{TAB: VALID}, we can see that the modeling error with respect to the $SP_N$ equations indeed stay at a relative low level with the selected coefficient set. The reason that such modeling error is not converging to zero might partly attribute to the simplification $\sigma_n \simeq (1-g^n)\sigma_{s}$ in the computation instead of using $\sigma_a + (1-g^n)\sigma_{s}$.

It is also informative to look at the convergence rate of $\phi_0^N$ from the $SP_N$ equations with respect to growing $N$, see~Fig~\ref{FIG: DECAY}. If the error converges sufficiently fast (e.g. exponentially), then we will obtain the uniqueness of reconstruction for the case that $N = \infty$. However, the theory about the convergence is still an open problem.

\begin{figure}[!htb]
    \centering
    \includegraphics[scale=0.7]{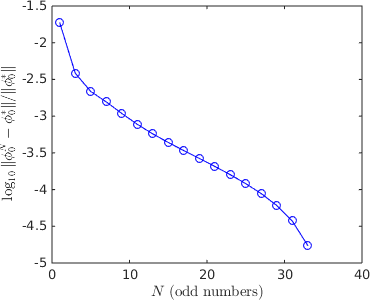}
     \includegraphics[scale=0.7]{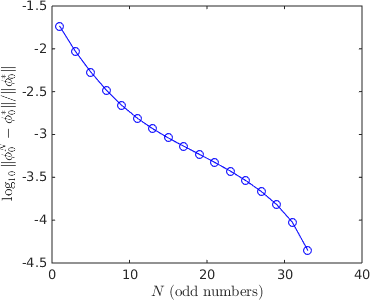}
    \caption{The convergence of $\phi_0^N$ with respect to the order $N$. Left: with source function $f_1$. Right: with source function $f_2$.}
    \label{FIG: DECAY}
\end{figure}
\subsection{Reconstruction of $\sigma_a$ only}
\label{SEC: REC A}
In this numerical experiment, we are using the algorithm introduced in Section~\ref{SEC:Recon} to reconstruct the absorption coefficient $\sigma_a$ only. We consider two scenarios for the reconstruction: (i) The datum $H$ is generated from certain $SP_N$ model. (ii) The datum $H$ is generated from the radiative transport model. The result is summarized in the following Tab~\ref{TAB: ABS REC} and Tab~\ref{TAB: ABS REC2}. For all the reconstructions in the tables below, we have contaminated the datum $H$ with multiplicative random noises pointwisely by $H^{\ast} = H (1 + \gamma\texttt{random})$ with parameter $\gamma = 5\%$ regarded as the noise level and $\texttt{random}$ is the uniform distributed random variable on $[-1,1]$.

\begin{table}[!htb]
    \centering
    \caption{Relative $L^2$ error of the reconstructed $\sigma_a$ for source function $f_1$ with different generating models and reconstruction models. The row label represents the generating model and the column label represents the reconstruction model.}
    \label{TAB: ABS REC}
    \begin{tabular}{||c|c|c|c|c|c|c|c|c|c||}
        \hline
        &$SP_1$  & $SP_3$ &  $SP_5$ & $SP_7$ &   $SP_9$ & $SP_{11}$ & $SP_{13}$ & $SP_{15}$ & $SP_{17}$ \\
        \hline
        $SP_1$&{\bf{2.89\%}}& 3.13\%&3.15\%& 3.14\%& 3.14\%& 3.14\%& 3.14\%& 3.14\% & 3.14\%\\
        \hline
        $SP_3$&3.08\%& {\bf{2.88\%}}& {{2.88\%}}& 2.89\%& 2.89\%& 2.89\%& 2.89\%& 2.89\% & 2.89\%\\
        \hline
        $SP_5$&3.13\%& {{2.91\%}}& {\bf{2.91\%}}& 2.91\%& 2.91\%& 2.91\%& 2.91\%& 2.91\% & 2.91\%\\
        \hline
        $SP_7$&3.09\%& {{2.89\%}}& {{2.88\%}}& 2.88\%& {\bf{2.88\%}}& 2.88\%& 2.88\%& 2.88\% & 2.88\%\\
        \hline
        $SP_9$&3.06\%& {{2.86\%}}&  {{2.86\%}}& 2.86\%& {\bf{2.86\%}}& 2.86\%& 2.86\%& 2.86\% & 2.86\%\\
        \hline
        $SP_{11}$&3.07\%& 2.88\%&  {{2.88\%}}& 2.88\%& 2.88\%&2.88\%& {\bf{2.88\%}}& 2.88\% & 2.88\%\\
        \hline
        $SP_{13}$&3.11\%& 2.91\%& {{2.91\%}}& {{2.91\%}}& 2.91\%& 2.91\%&{\bf{2.91\%}}& 2.91\% & 2.91\%\\
        \hline
        $SP_{15}$&3.09\%& 2.88\%&  {{2.88\%}}& 2.88\%& 2.88\%& 2.88\%& 2.88\%& {\bf{2.88\%}} & 2.88\%\\
        \hline
        $SP_{17}$&3.10\%& 2.89\%& {{2.89\%}}& 2.89\%& {{2.89\%}}& 2.89\%& 2.89\%& {\bf{2.89\%}}& {{2.89\%}}\\
        \hline\hline
        $RTE$&{{3.58\%}}& {\bf{3.14\%}}& {{3.16\%}}& 3.18\%& 3.19\%& 3.19\%& 3.19\%& 3.19\% & 3.20\%\\
        \hline
    \end{tabular}
\end{table}

\begin{table}[!htb]
    \centering
    \caption{Same as Tab~\ref{TAB: ABS REC}, but for source function $f_2$.}
    \label{TAB: ABS REC2}
    \begin{tabular}{||c|c|c|c|c|c|c|c|c|c||}
        \hline
        &$SP_1$  & $SP_3$ &  $SP_5$ & $SP_7$ &   $SP_9$ & $SP_{11}$ & $SP_{13}$ & $SP_{15}$ & $SP_{17}$ \\
        \hline
        $SP_1$&{\bf{2.91\%}}& 3.17\%& 3.22\%& 3.22\%& 3.22\%& 3.22\%& 3.22\%& 3.22\% & 3.22\%\\
        \hline
        $SP_3$&3.10\%& {\bf{2.88\%}}& 2.89\%& 2.90\%& 2.90\%&2.90\%& 2.90\%& 2.90\% &2.90\%\\
        \hline
        $SP_5$&3.14\%& 2.89\%& {\bf{2.89\%}}& 2.89\%& 2.89\%& 2.89\%& 2.89\%& 2.89\% & 2.89\%\\
        \hline
        $SP_7$&3.15\%& 2.89\%& {\bf{2.88\%}}& 2.88\%& 2.88\%& 2.88\%& 2.89\%& 2.89\% & 2.89\%\\
        \hline
        $SP_9$&3.16\%& 2.91\%& 2.89\% & {\bf{2.89\%}}& 2.89\%& 2.89\%& 2.89\%& 2.89\% & 2.89\%\\
        \hline
        $SP_{11}$&3.14\%& 2.89\%&  2.87\%& {\bf{2.87\%}}& 2.87\%& 2.87\%& 2.87\%& 2.87\% & 2.87\%\\
        \hline
        $SP_{13}$&3.12\%& 2.88\%& 2.87\%& {\bf{2.87\%}}& {{2.87\%}}& 2.87\%& 2.87\%& 2.87\% & 2.87\%\\
        \hline
        $SP_{15}$&3.17\%& 2.91\%&  2.90\%& {\bf{2.89\%}}& {{2.89\%}}& 2.89\%& 2.89\%& 2.89\% & 2.89\%\\
        \hline
        $SP_{17}$&3.16\%&2.90\%& 2.89\%& {\bf{2.89\%}}& {{2.89\%}}& 2.89\%& 2.89\%& 2.89\% & 2.89\%\\
        \hline\hline
        $RTE$&8.02\%& 7.12\%& {{6.77\%}}&{{6.65\%}}& 6.60\%& {\bf{6.59\%}}& 6.60\%& 6.60\% & 6.61\%\\
        \hline
    \end{tabular}
\end{table}
Observe the diagonals of above tables, one can find out that the reconstruction error is not growing as $N$ grows, this is because the estimate in Theorem~\ref{THM:1} is only meant for the worst boundary source. For the given source function and coefficient set, the reconstruction based on $SP_1$ model (diffusion approximation) appears not as good as the other models when the data are generated from $SP_N$ models. While the performances of most low order $SP_N$ models ($N\le 7$) are already close to the ones of high order $SP_N$ models ($N\ge 9$).

Particularly, when the datum $H$ is generated from the radiative transport model, all of the reconstruction errors of $SP_N$ models become larger due to the additional modeling errors, see Section~\ref{SEC: VALID APPOX}. We plot some of the reconstructions in Fig~\ref{FIG: ABS REC}. It is not surprising to find that the errors are relatively larger near boundary since the $SP_N$ equation system is still elliptic over the whole domain, while the behavior of radiative transport averaged solution is hyperbolic in the vicinity of boundary sets.

\begin{figure}[!htb]
    \centering
    \includegraphics[scale=0.393]{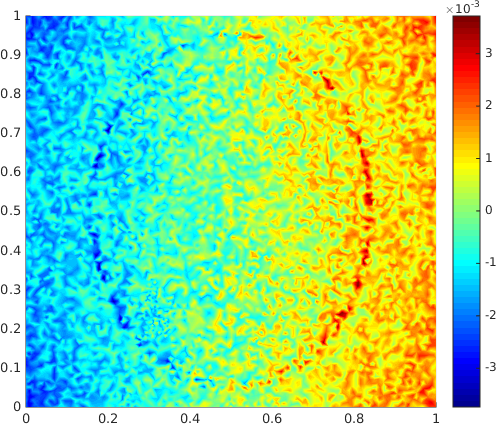}\,
    \includegraphics[scale=0.393]{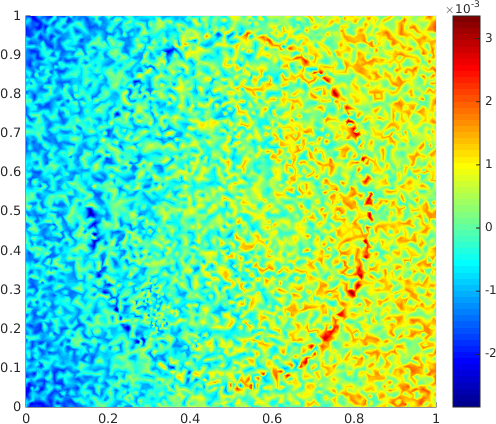}\,
    \includegraphics[scale=0.393]{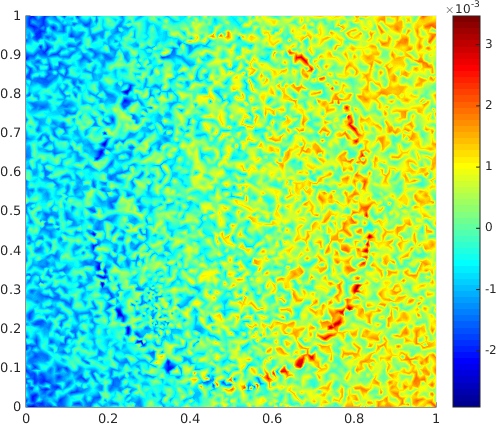}\\
    \includegraphics[scale=0.393]{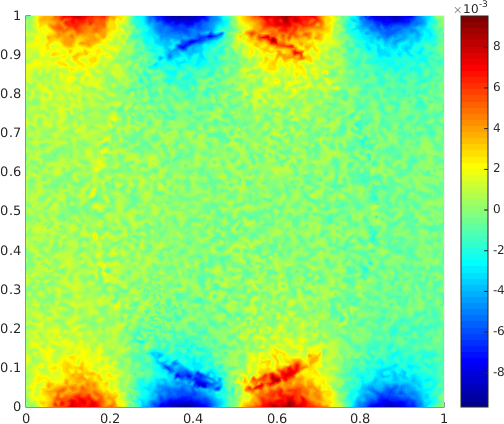}\,
    \includegraphics[scale=0.393]{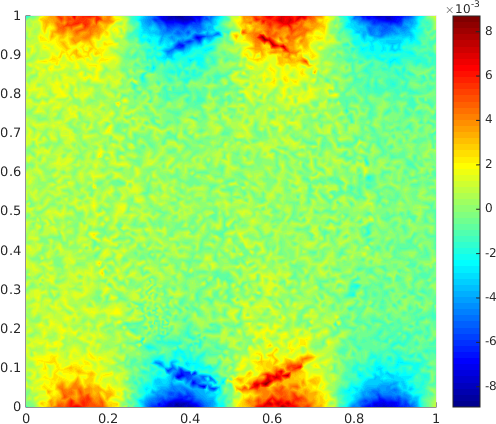}\,
    \includegraphics[scale=0.393]{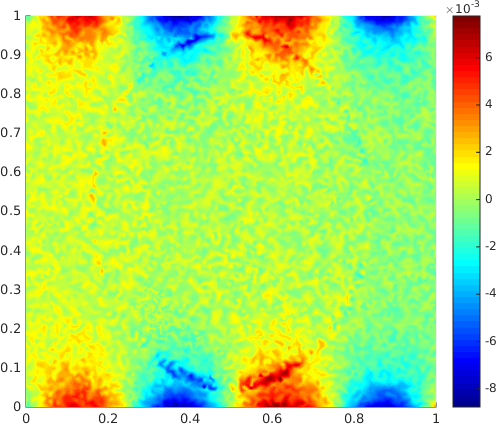}
    \caption{The error of reconstruction for the absorption coefficient $\sigma_a$ with datum generated from radiative transport model. Top row, from left to right: reconstruction error for source $f_1$ with $N=1$, $N=7$, $N=17$. Bottom row is the same but for source $f_2$.}
    \label{FIG: ABS REC}
\end{figure}
\subsection{Reconstruction of $\sigma_s$ only}
\label{SEC: REC S}
The reconstruction of $\sigma_s$ is a nonlinear problem for $N\ge 3$. In order to provide a fair comparison of across the $SP_N$ models, we will use the optimization based method to reconstruct the scattering coefficient:
\begin{equation}\label{EQ: H1 MIN}
\sigma_s = \argmin_{\sigma_s} \frac{1}{2}\int_{\Omega} \left( |H(\bx) - H^{\ast}|^2 + |\nabla (H(\bx) - H^{\ast})|^2 \right) d\bx  +\frac{\beta}{2} \int_{\Omega} |\nabla \sigma_s|^2 d\bx\,,
\end{equation}
where $H^{\ast}$ is the measured datum and $\beta$ is the regularization parameter. The optimization problem is solved by L-BFGS method, the gradient is computed from the adjoint state technique. The choice of regularization parameter $\beta$ should depend on the noise level. One should notice that we use $H^1(\Omega)$ norm instead of the traditional $L^2(\Omega)$ minimization:
\begin{equation}\label{EQ: L2 MIN}
\sigma_s = \argmin_{\sigma_s} \frac{1}{2}\int_{\Omega}  |H(\bx) - H^{\ast}|^2 d\bx  +\frac{\beta}{2} \int_{\Omega} |\nabla \sigma_s|^2 d\bx\,.
\end{equation}
This is due to the stability estimate in Theorem~\ref{THM: SCATTER}, where actually the $H^2(\Omega)$ norm is needed for the stability estimate. However, such regularity requirement implies that $H(\bx)$ needs to be globally $C^1$ from Sobolev embedding, which means our finite element space needs to equip with polynomials of five or higher degrees (e.g.  Argyris element). Here we relax objective functional to $H^1(\Omega)$ norm simply to avoid the extraordinary computational cost. To get a brief impression about the two optimization schemes, we take the $f_1$ source function with $SP_3$ model (both data generation and reconstruction) for an example, the datum $H$ is not contaminated (noise level $\gamma = 0$)
and regularization parameter $\beta = 0$ as well. The reconstructed scattering coefficients are shown in Fig~\ref{FIG: COMPARE SCHEME}. One can tell from the images that the coefficient recovered from $L^2$ optimization~\eqref{EQ: L2 MIN} still contains background artifacts. The reason behind is the relatively strong smoothing effect of the mapping $\sigma_s(\bx)\mapsto H(\bx)$,  where the high frequency information in $\sigma_s$ could not be fully recovered if we emphasize equally on $H(\bx)$'s frequency information. 

\begin{figure}[!htb]
    \centering
    \includegraphics[scale=0.45]{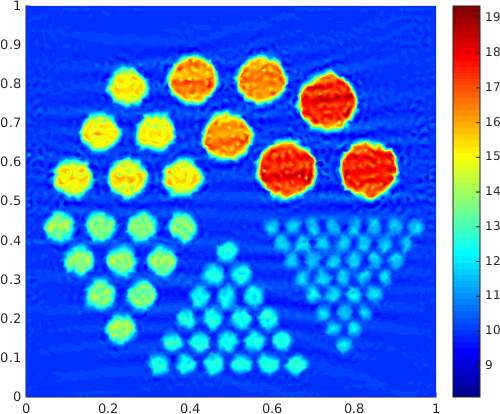}\quad
    \includegraphics[scale=0.45]{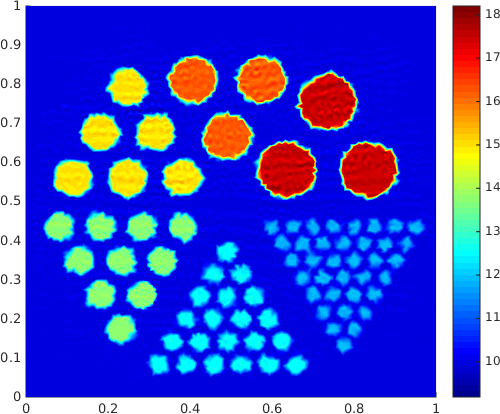}
    \caption{Left: The reconstructed scattering coefficeint using $L^2$ minimization, the $L^2$ relative error is $4.25\%$. Right: The reconstructed scattering coefficient using $H^1$ minimization, the $L^2$ relative error is $1.10\%$.}
    \label{FIG: COMPARE SCHEME}
\end{figure}

Therefore the data contamination should be treated carefully for the $H^1$ minimization~\eqref{EQ: H1 MIN}. This is because if we still apply the random noises by $H^{\ast} = H(1 + \gamma \texttt{random})$ at each mesh node, then $\|H - H^{\ast}\|_{H^1(\Omega)} = \cO(\sqrt{N}\gamma)$, where $N$ is the total number of nodes assuming the mesh is uniform. Therefore instead of pointwise multiplicative noise, we aggressively contaminate the datum by perturbing its Fourier modes
\begin{equation}
\cF{H^{\ast}}(\bxi) = \cF{H}(\bxi)(1 + \gamma \texttt{random})\,,
\end{equation}
where we have used $\gamma = 2\%$ as the noise level parameter. Note that such noise will perturb the low frequency modes of $H$ which might cause severe global artifacts in the reconstruction.

\begin{table}[!htb]
    \centering
    \caption{Relative $L^2$ error of the reconstructed $\sigma_s$ for source function $f_1$ with different generating models and reconstruction models. The row label represents the generating model and the column label represents the reconstruction model.}
    \label{TAB: ABS SCA}
    \begin{tabular}{||c|c|c|c|c|c|c|c|c|c||}
        \hline
        &$SP_1$  & $SP_3$ &  $SP_5$ & $SP_7$ &   $SP_9$ & $SP_{11}$ & $SP_{13}$ & $SP_{15}$ & $SP_{17}$ \\
        \hline
        $SP_1$&{\bf{12.0\%}}& 19.0\%& 19.5\%& 19.6\%& 19.6\%& 19.6\%& 19.6\%& 19.7\% & 19.7\%\\
        \hline
        $SP_3$&13.9\%& {\bf{12.1\%}}& 12.5\%& 12.8\%& 13.0\%& 13.2\%& 13.3\%& 13.4\% & 13.5\%\\
        \hline
        $SP_5$&14.9\%&12.5\%& {\bf{12.3\%}}& 12.5\%& 12.7\%& 12.9\%& 13.0\%& 13.1\% & 13.2\%\\
        \hline
        $SP_7$&15.1\%& 13.3\%& 12.5\%& {{12.1\%}}& 11.9\%& 11.8\%& 11.7\%& 11.7\% & {\bf{11.6\%}}\\
        \hline
        $SP_9$&15.1\%& 13.3\%&  12.5\%& 12.3\%& {\bf{12.2\%}}& {{12.3\%}}& 12.3\%& 12.4\% & 12.4\%\\
        \hline
        $SP_{11}$&15.5\%& 14.4\%&  13.3\%& 12.7\%& 12.3\%& 12.0\%& {{11.8\%}}& 11.7\% &  {\bf{11.6\%}}\\
        \hline
        $SP_{13}$&15.1\%& 13.8\%& 12.8\%& 12.5\%& 12.3\%& 12.2\%& 12.2\%& 12.2\% &  {\bf{12.2\%}}\\
        \hline
        $SP_{15}$&15.7\%&15.1\%&  14.0\%& 13.2\%& 12.7\%& 12.3\%& 12.1\%& 12.0\% & {\bf{11.8\%}}\\
        \hline
        $SP_{17}$&15.2\%& 14.2\%& 13.1\%& 12.7\%& 12.4\%& 12.3\%& 12.2\%& 12.2\% & {\bf{12.2\%}}\\
        \hline\hline
$RTE$&25.8\%& {{25.7\%}}& {\bf{18.3\%}}& 21.5\%&36.9\%& 37.2\%& 39.2\%& 39.3\% &{{39.4\%}}\\
\hline
    \end{tabular}
\end{table}

\begin{table}[!htb]
    \centering
    \caption{Same as Tab~\ref{TAB: ABS SCA} but for source function $f_2$.}
    \label{TAB: ABS SCA2}
    \begin{tabular}{||c|c|c|c|c|c|c|c|c|c||}
        \hline
        &$SP_1$  & $SP_3$ &  $SP_5$ & $SP_7$ &   $SP_9$ & $SP_{11}$ & $SP_{13}$ & $SP_{15}$ & $SP_{17}$ \\
        \hline
        $SP_1$&{\bf{20.8\%}}& 21.4\%& 21.3\%& 23.1\%& 21.4\%& 21.4\%& 21.5\%& 21.6\% & 21.7\%\\
        \hline
        $SP_3$&23.1\%& {\bf{18.0\%}}& 18.6\%& 19.1\%& 19.7\%& 19.9\%& 20.2\%& 20.3\% & 20.5\%\\
        \hline
        $SP_5$&26.9\%& 18.3\%& {\bf{17.7\%}}& 18.0\%& 18.3\%& 18.5\%& 18.6\%& 18.6\% & 18.7\%\\
        \hline
        $SP_7$&28.0\%& 22.7\%& 18.2\%& {{16.4\%}}& {\bf{16.3\%}}& 16.4\%& 16.6\%& 16.8\% & 16.9\%\\
        \hline
        $SP_9$&29.0\%& 19.9\%&  18.0\%&{{17.8\%}}& {\bf{17.7\%}}& 17.8\%& 17.9\%& 18.06\% & 18.1\%\\
        \hline
        $SP_{11}$&29.4\%& 19.5\%&  16.9\%& 16.4\%& {\bf{16.3\%}}& 16.3\%& 16.4\%& 16.4\% & 16.5\%\\
        \hline
        $SP_{13}$&29.5\%& 21.0\%& 18.4\%& 18.0\%& {\bf{17.8\%}}& 17.8\%& 17.8\%& 17.8\% &  {{17.9\%}}\\
        \hline
        $SP_{15}$&30.2\%& 20.6\%& 17.3\%& 16.6\%& 16.4\%& {\bf{16.3\%}}& 16.3\%& 16.3\% & {{16.4\%}}\\
        \hline
        $SP_{17}$&30.5\%&21.5\%&  19.4\%& 18.1\%& 17.9\%& 17.8\%& {\bf{17.7\%}}& 17.8\% & {{17.8\%}}\\
        \hline\hline
$RTE$&55.9\%& {\bf{47.8\%}}&  49.4\%&50.2\%& 50.3\%& 50.3\%& 51.4\% &{{51.5\%}}& {{51.6\%}}\\
\hline
    \end{tabular}
\end{table}

Similar to the previous numerical experiment, we summarize the result in the Tab~\ref{TAB: ABS SCA}, where we have fixed the regularization parameter $\beta = 10^{-8}$ for these experiments. From the table, we could clearly see that $SP_1$ is not as good as other models for the reconstruction of $\sigma_s$ when the datum $H$ is coming from higher order models. 

The reconstruction errors on the diagonal of the tables look converging as the order $N$ grows, which indicates that the models converge relatively fast for the given source functions and the coefficient setting.  When the datum $H$ is generated from the radiative transport equation, the reconstruction error becomes larger. The reconstructions with respect to source function $f_2$ are significantly worse than the ones for $f_1$, see Fig~\ref{FIG: SCA REC}. However, this could be explained through the analogue with the $SP_1$ model, where the scattering coefficient's reconstruction is to solve a transport equation~\cite{bal2011multi}:
\begin{equation}
\nabla \phi_0  \cdot \nabla D(x) + D(x) \Delta \phi_0 - \sigma_a \phi_0 = 0, 
\end{equation}
where $D(x) = 1/(3(1-g)\sigma_s)$, where $\phi_0$ and $\sigma_s|_{\partial\Omega}$ are known. Therefore when $\nabla \phi_0\neq 0$, the function $D(x)$ could be solved by tracing the characteristics. If $\nabla\phi_0$ vanishes or appears to be small, then the characteristics could be trapped, where the reconstructions are based on regularization only. 
\begin{figure}[!htb]
    \centering
    \includegraphics[scale=0.393]{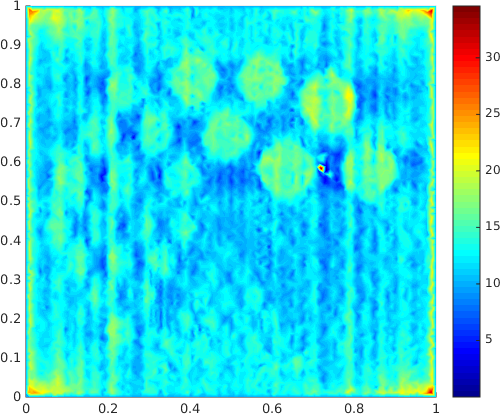}\,
    \includegraphics[scale=0.393]{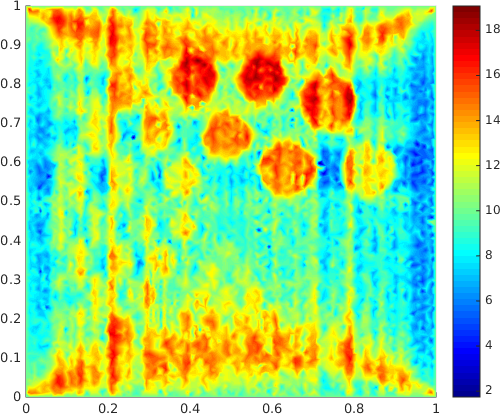}\,
    \includegraphics[scale=0.393]{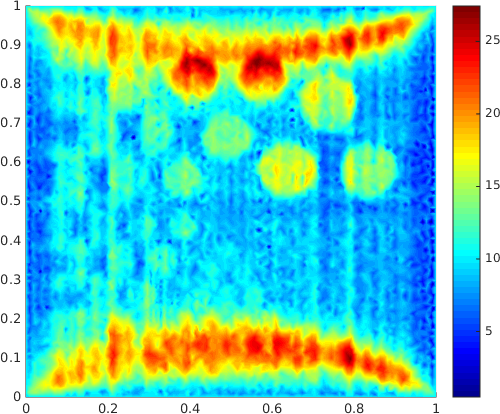}\\
    \includegraphics[scale=0.393]{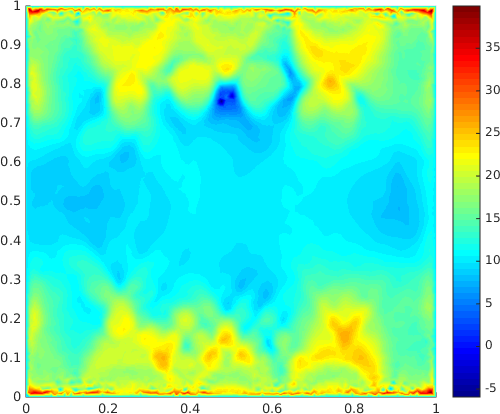}\,
    \includegraphics[scale=0.393]{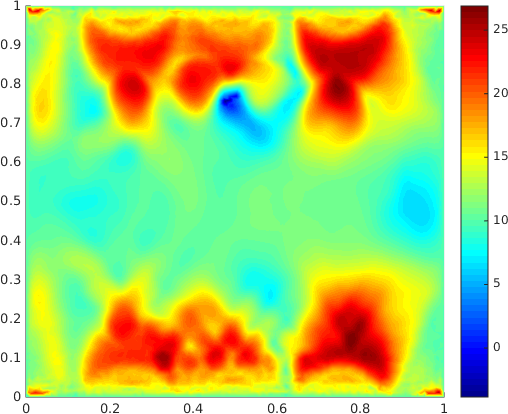}\,
    \includegraphics[scale=0.393]{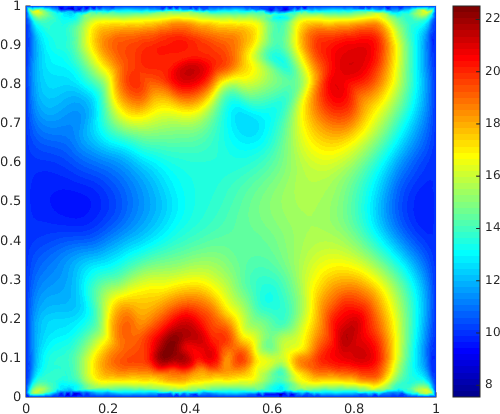}
    \caption{The reconstructions for the scattering coefficient $\sigma_s$ with datum generated from radiative transport model. Top row, from left to right: reconstruction error for source $f_1$ with $N=1$, $N=7$, $N=17$. Bottom row is the same but for source $f_2$.}
    \label{FIG: SCA REC}
\end{figure}

\subsection{Reconstruction of $\sigma_a$ and $\Upsilon$}
\label{SEC: REC A U}
In this section, we consider the non-linearized case and follow the aforementioned two-step reconstruction strategy. Suppose $H_1, H_2$ are the data sets measured with boundary source functions $f_1$ and $f_2$, respectively. Our numerical reconstruction solves the optimization problem:
\begin{equation}
\sigma_a^{\ast} = \argmin_{\sigma_a} \Big\|\frac{H_1}{H_2}- \frac{\bs_1\cdot \Phi_1}{\bs_1 \cdot \Phi_2}\Big\|_{L^2(\Omega)}^2 + \alpha \|\nabla \sigma_a\|_{L^2(\Omega)}^2\,,
\end{equation}
where $\Phi_1, \Phi_2$ are the solutions to the $SP_N$ equation with the absorption coefficient $\sigma_a$ and source function $f_1, f_2$, respectively. Intuitively, the ratio ${H_1}/{H_2}$ should be quite smooth and weakly depends on $\sigma_a$, which means the reconstruction for $\sigma_a$ could be very unstable. In the following, we assume the data sets $H_1$ and $H_2$ are generated from the $SP_N$ models with multiplicative noise $H_i^{\ast} = H_i (1 + \gamma \texttt{random})$ for $\gamma = 0.1\%$ only, the regularization parameter is fixed as $\alpha = 10^{-8}$. Then we reconstruct the absorption coefficient using the same model. The reconstructions are shown in Fig~\ref{FIG: SIMUL SIGMA}, it could be seen that the reconstructions are very unstable even for small noise, only limited resolution could be obtained. 

\begin{figure}[!htb]
    \centering
    \includegraphics[scale=0.38]{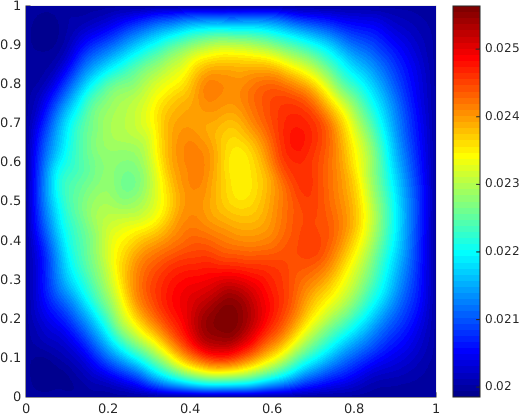}\,
    \includegraphics[scale=0.38]{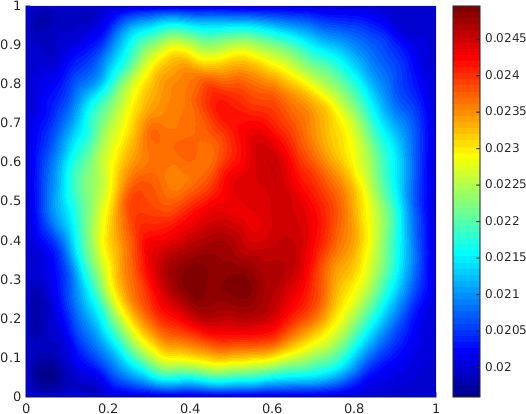}\,
    \includegraphics[scale=0.38]{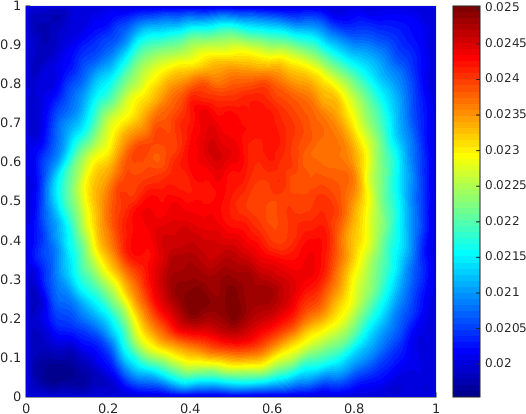}
    \caption{The reconstructions of absorption coefficients $\sigma_a$ with respect to $SP_N$ models. From left to right: $N = 1, 7, 17$. The relative $L^2$ errors are $16.6\%, 16.2\%, 15.9\%$, respectively.}
    \label{FIG: SIMUL SIGMA}
\end{figure}

\subsection{Reconstruction of $\sigma_a$ and $\sigma_s$}
\label{SEC: REC A S}
Similar to the previous case, we consider the reconstruction in a two-step process as well. Suppose  $H_1, H_2$ are the data sets with boundary sources $f_1, f_2$, respectively. Our algorithm will first construct the scattering coefficient from the following optimization problem:
\begin{equation}
\sigma_s^{\ast} = \argmin_{\sigma_s} \Big\|\frac{H_1}{H_2}- \frac{\bs_1\cdot \Phi_1}{\bs_1 \cdot \Phi_2}\Big\|_{H^1(\Omega)}^2 + \alpha \|\nabla \sigma_s\|_{L^2(\Omega)}^2\,,
\end{equation}
where $\Phi_1,\Phi_2$ are the solutions to the modified $SP_N$ equation~\eqref{EQ:BILINEAR2}, where $\sigma_a$ has been replaced by $H_i/(\Upsilon \bs_1\cdot \Phi_i)$, $i=1,2$. Here we have taken the $H_1$ minimization.  In the following numerical experiment, we assume $H_1, H_2$ are generated from the $SP_N$ model, the data sets are contaminated on the Fourier space through

$$\cF H^{\ast}_i(\bxi) = \cF \cH_i(\bxi)(1 + \gamma \texttt{random})$$
with $\gamma = 2\%$.
We also fix the regularization parameter $\alpha = 10^{-8}$.The numerical reconstructions are performed over the same $SP_N$ model and the results are illustrated in Fig~\ref{FIG: SIMUL SIGMS}. After the scattering coefficient has been reconstructed, we will use the recovered scattering coefficient to find the absorption coefficient following the Experiment 4.4. The corresponding reconstruction errors of the absorption coefficients are shown in Fig~\ref{FIG: SIMUL SIGMS A}. It can be seen that even the reconstruction of scattering coefficients contain background artifacts, while the reconstruction errors of $\sigma_a$ are still quite small.

\begin{figure}[!htb]
    \centering
    \includegraphics[scale=0.39]{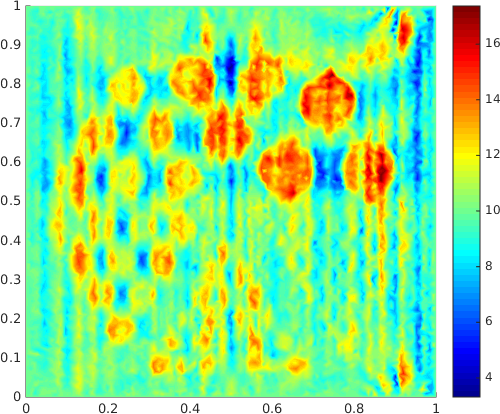}\,
    \includegraphics[scale=0.39]{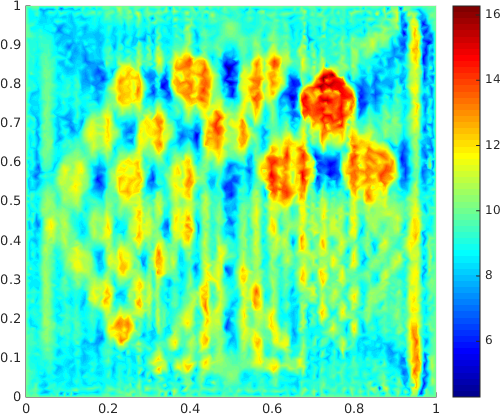}\,
    \includegraphics[scale=0.39]{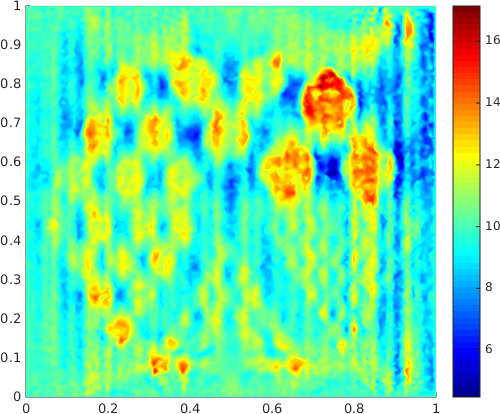}
    \caption{The reconstructions of scattering coefficients $\sigma_s$ with respect to $SP_N$ models. From left to right: $N = 1, 7, 17$. The relative $L^2$ errors are $16.1\%, 16.7\%, 16.8\%$, respectively.}
    \label{FIG: SIMUL SIGMS}
\end{figure}

\begin{figure}[!htb]
    \centering
    \includegraphics[scale=0.39]{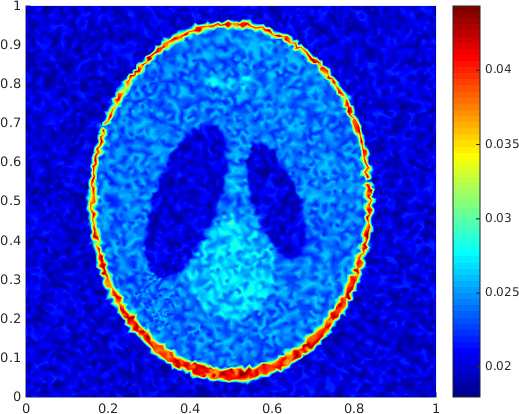}
    \includegraphics[scale=0.39]{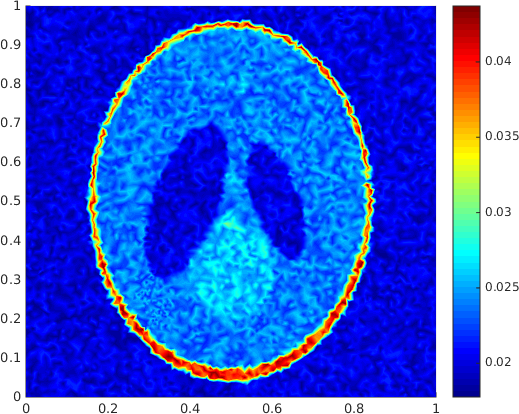}
    \includegraphics[scale=0.39]{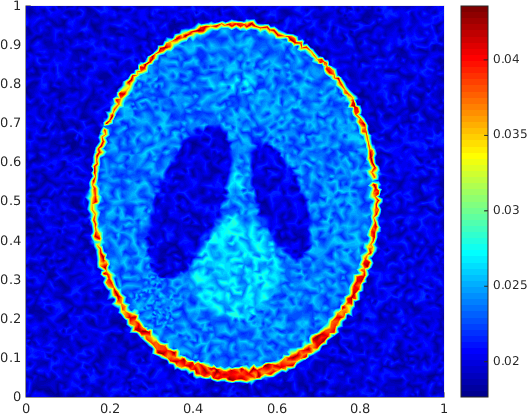}
    \caption{The reconstruction error of absorption coefficients $\sigma_a$ with respect to $SP_N$ models by using the recovered scattering coefficients from Fig~\ref{FIG: SIMUL SIGMS}. From left to right: $N = 1, 7, 17$. The relative $L^2$ errors are $1.76\%, 5.79\%,5.80\%$, respectively.}
    \label{FIG: SIMUL SIGMS A}
\end{figure}

\section{Conclusion}
\label{SEC:Concl}
In this work, we studied the quantitative photoacoustic tomography with the simplified $P_N$ approximation model to the radiative transport equation. We have derived the uniqueness and stability estimates for the reconstruction of one single coefficient of $(\sigma_a$, $\sigma_s$, $\Upsilon$) from one initial pressure datum $H(\bx)$. For the simultaneous reconstruction of two coefficients, we have considered the linearized setting and introduced the optimization based numerical algorithm for the reconstruction. We showed the numerical simulations based on a synthetic data to validate the mathematical analysis. 
\section*{Acknowledgment}
H. Zhao's research is partially supported by NSF DMS-2048877 and DMS-2012860.

\appendix
\section{Appendix}
\label{SEC: APP}
\begin{lemma}\label{LEM: M INV}
    Let $s_{k,l}$ be the $(k,l)$-th entry of $M^{-1}$, then 
        \begin{equation}
    s_{k, l} = 
    \begin{cases}
    \frac{1}{2k-1}(-1)^{l-k}\frac{((2l-2))!!}{(2l-1)!!}\frac{(2k-1)!!}{(2k-2)!!},\quad &l\ge k\,,\\
    0, &\text{otherwise}\,.
    \end{cases} 
    \end{equation}
\end{lemma}
\begin{proof}
Let $S$ be the matrix with $(k,l)$-th entry as $s_{k,l}$, we then compute the $(k,j)$-th entry of $SM$ by
$a_{k,j} := s_{k, j}(2j-1)+ s_{k, j-1} (2j-2)$. If $k=j$, $s_{k,j} = \frac{1}{2j-1}$, therefore $a_{k,j} = \frac{1}{2j-1} (2j-1) = 1$. If $j<k$, notice that now $s_{k,j} = 0$, we must have $a_{k,j}  = 0$. If $j > k$, we compute $a_{k,j}$ directly
\begin{equation}
\frac{1}{2k-1}\frac{(2k-1)!!}{(2k-2)!!} \left( (-1)^{j-k}\frac{((2j-2))!!}{(2j-1)!!}(2j-1) + (-1)^{j-k-1}\frac{((2j-4))!!}{(2j-3)!!} (2j-2)\right) = 0\,.
\end{equation}
Hence $S = M^{-1}$.
\end{proof}
\begin{lemma}\label{LEM:DET}
$\det(R^{-1})\le \cO(N^{3/8})$, hence $\det(R) \ge \cO(N^{-3/8})$.
\end{lemma}
\begin{proof}
	Since 
	\begin{equation}
	\begin{aligned}
	R_{ij} &= (-1)^{i+j-1}\frac{2\Gamma(i+\frac{1}{2})\Gamma(j-\frac{1}{2})}{\pi \Gamma(i)\Gamma(j)} \frac{(4j-3)}{ (2i+2j-2)(2j-2i -1) }\,,\\
	&= (-1)^{i+j}\frac{\Gamma(i+\frac{1}{2})\Gamma(j-\frac{1}{2})}{\pi \Gamma(i)\Gamma(j)} \frac{2j-\frac{3}{2}}{(i-\frac{1}{4})^2 - (j - \frac{3}{4})^2}\,,
	\end{aligned}
	\end{equation}
	which gives the factorization
	\begin{equation}
	R = U G V\,,
	\end{equation}
	where diagonal matrices 
    \begin{equation}
    \begin{aligned}
    U &= \diag(u_1,\dots, u_{(N+1)/2})\,,\\ V &= \diag(v_1,\dots, v_{(N+1)/2})\,,
    \end{aligned}
    \end{equation}
    with $  u_i = (-1)^{i}\sqrt{\frac{2}{\pi}}\frac{\Gamma(i+\frac{1}{2})}{\Gamma(i)}, v_j = (-1)^j\sqrt{\frac{2}{\pi}} \frac{\Gamma(j-\frac{1}{2})}{\Gamma(j)}(j-\frac{3}{4})\,,$
    and $G$ is the Cauchy-Toeplitz matrix,
	\begin{equation}
G = \left( \frac{1}{(i-\frac{1}{4})^2 - (j - \frac{3}{4})^2}\right)_{i,j=1,\dots ,(N+1)/2}\,.
	\end{equation}
	We denote $x_i = (i-\frac{1}{4})^2$ and $y_j = (j - \frac{3}{4})^2$, then use the Theorem 2.1 in~\cite{tyrtyshnikov1992singular}, 
	\begin{equation}
	G^{-1} = P G^T Q\,,
	\end{equation}
	where ${P} = \text{diag}(p_1,\dots, p_{(N+1)/2})$ and $Q = \text{diag}(q_1, \dots, q_{(N+1)/2})$ satisfy
	\begin{equation}
	\begin{aligned}
	G [p_1, \dots, p_{(N+1)/2}]^T &= [1, \dots, 1]^T, \\
	G^T [q_1, \dots, q_{(N+1)/2}]^T &= [1, \dots, 1]^T\,.
	\end{aligned}
	\end{equation}
	Also $p_i, q_j$ are computed explicitly by Cramer's law,
	\begin{equation}
	\begin{aligned}
	p_i &= \prod_{l=1}^{(N+1)/2}(x_l - y_i) /\prod_{l=1,l\neq i}^{(N+1)/2} (y_l - y_i)\,,\\
	q_j &= \prod_{l=1}^{(N+1)/2} (x_j - y_l) / \prod_{l=1,l\neq j}^{(N+1)/2} (x_j - x_l)\,.
	\end{aligned}
	\end{equation}
	By replacing $x_i, y_j$ with their values, we obtain
	\begin{equation}
	\begin{aligned}
	p_i &= \prod_{l=1}^{(N+1)/2}(l+i-1)(l-i+\frac{1}{2})  /\prod_{l=1,l\neq i}^{(N+1)/2} (l-i)(l+i-3/2) \\
	&= \left(\prod_{l=1}^{(N+1)/2}\frac{l+i-1}{l+i-\frac{3}{2}} \right) \note{\left(2i-\frac{3}{2}\right)}\left(\frac{1}{2}\prod_{l=1}^{i-1}\left(1 - \frac{1}{2l}\right) \prod_{l=1}^{(N+1)/2 - i} \left(1 + \frac{1}{2l}\right)\right) \\
	&=\frac{\Gamma(i + \frac{N+1}{2})}{\Gamma(i + \frac{N}{2})} \frac{\Gamma(i-\frac{1}{2})}{\Gamma(i)}\left(2i-\frac{3}{2}\right) \left(\frac{1}{2}\prod_{l=1}^{i-1}\left(1 - \frac{1}{2l}\right) \prod_{l=1}^{(N+1)/2 - i} \left(1 + \frac{1}{2l}\right)\right)\,,\\
	q_j &= \prod_{l=1}^{(N+1)/2} (j+l-1)(j-l+\frac{1}{2}) / \prod_{l=1,l\neq j}^{(N+1)/2} (j-l)(j + l-1/2)\\
	&= \left( \prod_{l=1}^{(N+1)/2} \frac{j+l-1}{j+l-\frac{1}{2}}\right)\note{\left(2j -\frac{1}{2}\right)}\left(\frac{1}{2}\prod_{l=1}^{j-1}\left(1+\frac{1}{2l}\right) \prod_{l=1}^{(N+1)/2 - j} \left(1-\frac{1}{2l}\right)\right) \\
	&= \frac{\Gamma(j + \frac{N+1}{2})\Gamma(j+\frac{1}{2})}{\Gamma(j + \frac{N}{2} + 1)\Gamma(j)}\left(2j -\frac{1}{2}\right)\left(\frac{1}{2}\prod_{l=1}^{j-1}\left(1+\frac{1}{2l}\right) \prod_{l=1}^{(N+1)/2 - j} \left(1-\frac{1}{2l}\right)\right)\,.
	\end{aligned}
	\end{equation}
	Then we can easily deduce $\det(R^{-1}) = \sqrt{\det(P)\det(Q)}/(\det(V)\det(U))$, since all matrices involved are diagonal, let $S_i =\frac{\pi}{2} \prod_{l=1}^{i-1}\left(1+\frac{1}{2l}\right) \prod_{l=1}^{(N+1)/2 - i} \left(1-\frac{1}{2l}\right)$, then from the theory of Gamma functions, we know
	\begin{equation}
	\begin{aligned}
	S_i = \frac{\Gamma(i+\frac{1}{2}) \Gamma(\note{\frac{N+1}{2}}-i+\frac{1}{2})}{\Gamma(i)\Gamma(\frac{(N+1)}{2}-i+1)}\,.
	\end{aligned}
	\end{equation}
	Hence we can estimate $\det(R^{-1})$'s upper bound by estimating
	\begin{equation}
	\begin{aligned}
	\det(R^{-1}) &= \prod_{i=1}^{(N+1)/2}\frac{\sqrt{p_i q_i}}{u_i v_i} \\
	&=\prod_{i=1}^{(N+1)/2} S_i \sqrt{\frac{ \Gamma(i+\frac{N+1}{2})^2 \Gamma(i)^2(i-\frac{1}{4})}{\Gamma(i+\frac{N}{2}) \Gamma(i+\frac{N}{2}+1) \Gamma(i+\frac{1}{2})\Gamma(i-\frac{1}{2})(i-\frac{3}{4}) }} \\
	&=\prod_{i=1}^{(N+1)/2} \frac{\Gamma(i-\frac{1}{2})}{\Gamma(i)}\sqrt{\frac{\Gamma(i+\hN)^2(i-\frac{1}{2})(i-\frac{1}{4})}{\Gamma(i+\hN -\frac{1}{2})\Gamma(i+\hN +\frac{1}{2})(i-\frac{3}{4})}} \,.
	\end{aligned}
	\end{equation}
	Then by noticing the Chu's Double Inequality~\cite{chu1962modified}, 
	\begin{equation}
	\sqrt{x - \frac{1}{4}} < \frac{\Gamma(x + \frac{1}{2})}{\Gamma(x)} < \frac{x}{\sqrt{x + \frac{1}{4}}}\,,
	\end{equation}
	the following estimates hold,
	\begin{equation}
	\begin{aligned}
\sqrt{\frac{ \Gamma(i+\frac{N+1}{2})^2 }{\Gamma(i+\frac{N}{2}) \Gamma(i+\frac{N}{2}+1)  }} &\le \sqrt{\frac{i + \frac{(N+1)}{2} - \frac{1}{2}}{i + \hN -\frac{1}{4} }} < 1\,, \\
\sqrt{\frac{\Gamma(i-\frac{1}{2})^2}{\Gamma(i)^2}\frac{(i-\frac{1}{2})(i-\frac{1}{4})}{i-\frac{3}{4}}} &\le \sqrt{\frac{(i-\frac{1}{2})(i-\frac{1}{4})}{(i-\frac{3}{4})^2}}\,,
	\end{aligned}
	\end{equation}
	which implies
	\begin{equation}
	\begin{aligned}
	\det(R^{-1}) &<  \prod_{i=1}^{(N+1)/2} \sqrt{\frac{(4i-2)(4i-1)}{(4i-3)^2}}\\&= \exp\left(\frac{1}{2}\sum_{i=1}^{(N+1)/2}\log\left(1 + \frac{3}{4i-3} + \frac{2}{(4i-3)^2}\right)\right)\\
	&\le \exp\left(\frac{1}{2}\sum_{i=1}^{(N+1)/2} \frac{3}{4i-3} + \frac{2}{(4i-3)^2}\right)\\
	&= \cO\left(\exp\left(\frac{3}{8}\log \left(\hN\right) \right)\right) = \cO(N^{3/8})\,.
	\end{aligned}
	\end{equation}
\end{proof}
\begin{lemma}\label{LEM:FRO}
	$\|R\|_F^2 \le \frac{N+1}{2}$.
	\end{lemma}
	\begin{proof}
		We show that for all $m \ge 1$,
		\begin{equation}
		\sum_{j = 1}^{m} |R_{j, m}|^2 + \sum_{j = 1}^{m-1} |R_{m, j}|^2  \le 1 \,.
		\end{equation} 
		The above estimate is true for $m=1$ since $R_{1,1} = \frac{1}{2}$, we only focus on the cases that $m\ge 2$.
		Reformulate $R_{j, m}$ by the Gamma function as
		\begin{equation}
		|R_{j,m}|  = \frac{\Gamma(j+\frac{1}{2})\Gamma(m-\frac{1}{2})}{\pi \Gamma(j)\Gamma(m) (j+m-1)} \frac{(4m-3)}{2m-2j -1 }\,.
		\end{equation}
		From the Chu's Double Inequality~\cite{chu1962modified}  
		\begin{equation}
		\sqrt{x - \frac{1}{4}} < \frac{\Gamma(x + \frac{1}{2})}{\Gamma(x)} < \frac{x}{\sqrt{x + \frac{1}{4}}}\,,
		\end{equation}
		the following estimates hold,
		\begin{equation}
		\frac{\Gamma(j+\frac{1}{2})}{\Gamma(j)} < \frac{j}{\sqrt{j+\frac{1}{4}}}\,,\quad \frac{\Gamma(m-\frac{1}{2})}{\Gamma(m)} < \frac{1}{(m-\frac{3}{4})^{1/2}}\,,
		\end{equation}
		we can deduce the estimate 
		\begin{equation}
		\begin{aligned}
		\sum_{j = 1}^{m} |R_{j, m}|^2 
		&= \frac{1}{\pi^2}\sum_{j=1}^m \left( \frac{\Gamma(j+\frac{1}{2})\Gamma(m-\frac{1}{2})}{ \Gamma(j)\Gamma(m) } \right)^2 \frac{(4m-3)^2}{(j+m-1)^2} \frac{1}{(2m-2j-1)^2} \\
		&\le \frac{1}{\pi^2} \sum_{j=1}^m \frac{(4m-3)^2 (j-\frac{1}{4} + \frac{1}{4(1+4j)})}{(m-\frac{3}{4})(j+m-1)^2 (2m-2j-1)^2} \\ 
		&\le \frac{4}{\pi^2}   \sum_{j=1}^m  \frac{(4j-1 + \frac{1}{(1+4j)}) (4m-3)}{(2j+2m-2)^2 (2m-2j-1)^2} \\
		&= \frac{4}{\pi^2} \sum_{j=1}^m \left[\frac{1}{(2m-2j-1)^2} - \frac{1}{(2j+2m-2)^2}\right] \\&\quad+ \frac{4}{\pi^2} \sum_{j=1}^m \frac{(4m-3)}{(1+4j)(2j+2m-2)^2 (2m-2j-1)^2} \\
		&\le \frac{4}{\pi^2}\left[\left(1 + \frac{\pi^2}{8} -  \frac{1}{4m}\right)+\frac{1}{5 m(2m-3)^2} +  \frac{1}{9  m} \sum_{j=2}^m \frac{1}{ (2m-2j-1)^2} \right]\\
		&\le  \frac{4}{\pi^2} \left[\left(1 + \frac{\pi^2}{8}\right)+ \left(\frac{11}{180} + \frac{\pi^2}{72}\right) \frac{1}{m} \right]\,.
		\end{aligned}
		\end{equation}
		The other part can be estimated in a similar way,
		\begin{equation}
		\begin{aligned}
		\sum_{j=1}^{m-1} |R_{m,j}|^2 
		&=\frac{1}{\pi^2}\sum_{j=1}^{m-1} \left(\frac{\Gamma(m+\frac{1}{2})\Gamma(j-\frac{1}{2})}{ \Gamma(m)\Gamma(j)} \right)^2 \frac{(4j-3)^2}{(j+m-1)^2} \frac{1}{(2j-2m-1)^2} \\
		&\le \frac{1}{\pi^2} \sum_{j=1}^{m-1} \frac{(4j-3)^2 (m-\frac{1}{4} + \frac{1}{4(1+4m)})}{(j-\frac{3}{4})(j+m-1)^2 (2j-2m-1)^2} \\
		&\le \frac{4 }{\pi^2}\sum_{j=1}^{m-1} \frac{(4m-1+\frac{1}{1+4m})(4j-3)}{(2j+2m-2)^2 (2j-2m-1)^2} \\ 
		&= \frac{4}{\pi^2} \sum_{j=1}^{m-1}\left[ \frac{1}{(2m+1-2j)^2} - \frac{1}{(2j+2m-2)^2}  \right] \\
		&\quad + \frac{4}{\pi^2} \sum_{j=1}^{m-1} \frac{4j-3}{(1+4m)(2j+2m-2)^2(2m+1 - 2j)^2} \\
		&\le \frac{4}{\pi^2}\left[ \left( \frac{\pi^2}{8} - 1 - \frac{1}{4m}  \right) + \sum_{j=1}^{m-1}\frac{1}{4m^2(2m+1 - 2j)^2} \right]\\
		&\le \frac{4}{\pi^2}\left[\left( \frac{\pi^2}{8} - 1 -\frac{1}{4m} \right)+\left(\frac{\pi^2}{8} -1\right) \frac{1}{4m^2}\right]\,.\\
		\end{aligned}
		\end{equation}
		Hence we can estimate that
		\begin{equation}\label{EQ:INEQ LEM APDX}
		\sum_{j = 1}^{m} |R_{j, m}|^2 + \sum_{j = 1}^{m-1} |R_{m, j}|^2 \le 1 - \frac{4}{\pi^2}\left(\frac{17}{90} - \frac{\pi^2}{72}\right)\frac{1}{m} + \frac{1}{\pi^2}\left(\frac{\pi^2}{8} - 1\right) \frac{1}{m^2}\,,
		\end{equation}
		the above estimate is strictly less than $1$ for $m\ge 2$, then the Frobenius norm's square of $R$ is estimated by
		\begin{equation}
		\|R\|_{F}^2 = \frac{1}{2} + \sum_{m=2}^{(N+1)/2} \left(\sum_{j = 1}^{m} |R_{j, m}|^2 + \sum_{j = 1}^{m-1} |R_{m, j}|^2\right) \le \frac{N}{2}\,.
		\end{equation}
		\end{proof}
    \begin{lemma}\label{LEM: SQ BD}
        Let $\bs_k$ the $k$-th row of the matrix \note{$M^{-1}$} in~\eqref{EQ:T MAT}, and matrix $\bQ$ is defined in~\eqref{EQ: P Q}, then 
        \begin{equation}
        \begin{aligned}
        &\|\bs_1\cdot \bQ\|_{\ell^p} = \cO(N^{3/2+1/p})\,.
        \end{aligned}
        \end{equation}
    \end{lemma}
    \begin{proof}
        Let $s_{k,j}$ denote the $j$-th entry of vector $\bs_k$. Combine the Lemma~\ref{LEM: M INV} and Gautschi's inequality~\cite{gautschi1959some}, we have
        \begin{equation}
        |s_{k,j}| < \frac{1}{2k-1}\frac{\sqrt{k+\frac{1}{2}}}{\sqrt{j-\frac{1}{2}}}\,.
        \end{equation}
        For the matrix $\bQ$,  we denote its $(n,j)$-th entry by $Q_{nj}$, which can be estimated by
        \begin{equation}
        \begin{aligned}
           |Q_{nj}| &\le (4n-1)\sum_{k\ge 2}^{\min(j, n)} (4k-3) |s_{k, n} s_{k, j} |\le (4n-1) \sum_{k\ge 2}^{\min(j, n)}  \frac{(4k-3)}{(2k-1)^2}\frac{k+\frac{1}{2}}{\sqrt{j-\frac{1}{2}}\sqrt{n-\frac{1}{2}}} \\
           &\le \frac{(4n-1)}{\sqrt{j-\frac{1}{2}}\sqrt{n-\frac{1}{2}}} \sum_{k\ge 2}^{\min(j, n)} \frac{(4k-3)(k+\frac{1}{2})}{(2k-1)^2}\,.
        \end{aligned}
        \end{equation}
        Since $(4k-3)(k+\frac{1}{2}) \le  2(2k-1)^2$ for all $k$, we will have $|Q_{n,j}|\le \frac{(8n-2)}{\sqrt{n-\frac{1}{2}}\sqrt{j-\frac{1}{2}}}(\min(j,n)-1)$, therefore, by simple calculations, $\|\bs_1\cdot \bQ\|_{\ell^{\infty}}$ is bounded by
        \begin{equation}
        \begin{aligned}
         \|\bs_1\cdot \bQ\|_{\ell^{\infty}} &\le \sup_{j\ge 1} \sum_{n\ge 1} |s_{1, n}| |Q_{n,j}| \le  \sup_{j\ge 1}\sum_{n \ge 1} \frac{(8n-2)(\min(j,n)-1) }{(n-\frac{1}{2})\sqrt{j-\frac{1}{2}}}\\
         & = \cO(N^{3/2})\,.
        \end{aligned}
        \end{equation}
        Then the $\ell^p$ estimate is $$\|\bs_1\cdot \bQ\|_{\ell^{p}}\le  \|\bs_1\cdot \bQ\|_{\ell^\infty} \left(\hN\right)^{1/p} = \cO(N^{3/2+1/p})\,. $$
    \end{proof}
\bibliographystyle{siam}
\bibliography{main}
\end{document}